\documentclass[10pt]{amsart}
\let\oldmarginpar\marginpar
\renewcommand\marginpar[1]{\-\oldmarginpar[\raggedleft\footnotesize #1]%
{\raggedright\footnotesize #1}}

\usepackage{mathptmx}
\usepackage{amsmath}
\usepackage{amscd}
\usepackage{amssymb}
\usepackage{amsthm}
\usepackage{xspace}
\usepackage[all,tips]{xy}
\usepackage[dvips]{graphicx}
\usepackage{verbatim}
\usepackage{syntonly}
\usepackage{hyperref}
\usepackage{arydshln}
\usepackage{lineno}

\theoremstyle{plain}
\newtheorem{thm}{Theorem}[section]
\newtheorem{cor}[thm]{Corollary}
\newtheorem{prop}[thm]{Proposition}
\newtheorem{lemma}[thm]{Lemma}

\newtheorem{defn}[thm]{Definition}

\newtheoremstyle{TheoremNum}
{\topsep}{\topsep}
{\itshape}
{}
{\bfseries}
{.}
{ }
{\thmname{#1}\thmnote{ \bfseries #3}}
\theoremstyle{TheoremNum}
\newtheorem{thmn}{Theorem}

\DeclareMathOperator{\lcm}{lcm}

\DeclareMathOperator{\D}{D}

\DeclareMathOperator{\Farb}{F}

\DeclareMathOperator{\RD}{RF}

\DeclareMathOperator{\RFU}{RFU}
\DeclareMathOperator{\RFL}{RFL}
\DeclareMathOperator{\ARF}{ARF}

\DeclareMathOperator{\LR}{LR}
\DeclareMathOperator{\UR}{UR}

\DeclareMathOperator{\RP}{RP}

\newcommand{\bdef}{\overset{\text{def}}{=}}

\newcommand{\al}{\alpha}

\newcommand{\ga}{\gamma}

\newcommand{\la}{\lambda}

\newcommand{\innp}[1]{\left< #1 \right>}

\newcommand{\set}[1]{\left\{#1\right\}}

\newcommand{\pr}[1]{\left( #1 \right) }

\newcommand{\Fr}[1]{\ensuremath{\mathfrak{#1}}}

\newcommand{\N}{\ensuremath{\mathbb{N}}}
\newcommand{\Q}{\ensuremath{\mathbb{Q}}}
\newcommand{\R}{\ensuremath{\mathbb{R}}}
\newcommand{\Z}{\ensuremath{\mathbb{Z}}}

\newcommand{\nsub}{\trianglelefteq}
\newcommand{\map}[3]{#1 : #2 \rightarrow #3}
\newtheorem{innercustomgeneric}{\customgenericname}
\providecommand{\customgenericname}{}
\newcommand{\newcustomtheorem}[2]{%
	\newenvironment{#1}[1]
	{%
		\renewcommand\customgenericname{#2}%
		\renewcommand\theinnercustomgeneric{##1}%
		\innercustomgeneric
	}
	{\endinnercustomgeneric}
}

\newcustomtheorem{customthm}{Theorem}
\newcustomtheorem{customproposition}{Proposition}

\newtheorem*{fake1}{\cite[\textbf{Theorem 1.1}]{Pengitore_1}}

\begin{document}

\title{\textbf{Residual dimension of nilpotent groups}}
\author{Mark Pengitore}

\maketitle

\begin{abstract}
The function $\Farb_{G}(n)$ gives  the maximum order of a finite group  needed to distinguish a nontrivial element of $G$ from the identity with a surjective group morphism as one varies over nontrivial elements of word length at most $n$. In previous work \cite{Pengitore_1}, the author claimed a characterization for $\Farb_{N}(n)$ when $N$ is a finitely generated nilpotent group. However, a counterexample to the above claim  was communicated to the author, and subsequently, the statement of the asymptotic characterization of $\Farb_{N}(n)$ is incorrect. In this article, we introduce new tools to provide lower asymptotic bounds for $\Farb_{N}(n)$ when $N$ is a finitely generated nilpotent group. Moreover, we introduce a class of finitely generated nilpotent groups for which the upper bound of \cite{Pengitore_1} can be improved. Finally, we construct of a class of finitely generated nilpotent groups $N$ for which the asymptotic behavior of $\Farb_N(n)$ can be fully characterized.
\end{abstract}

\section{Introduction}

A group $G$ is \textbf{residual finite} if for each nontrivial element $x \in G$, there exists a surjective group morphism $\map{\varphi}{G}{Q}$ to a finite group such that $\varphi(x) \neq 1$. When $G$ comes equipped with a finite generating subset $S$, we are able to quantify residual finiteness of $G$ with the function $\Farb_{G,S}(n)$ introduced by Bou-Rabee in \cite{Bou_Rabee10} which is the value of $\Farb_{G,S}(n)$ is the maximum order of a finite group  needed to distinguish a nontrivial element from the identity as one varies over nontrivial elements of word length at most $n$. Since the dependence of $\Farb_{G,S}(n)$ on $S$ is mild by a result in \cite{Bou_Rabee10}, we suppress the generating subset throughout the introduction.

In previous work \cite{Pengitore_1}, we claimed an effective characterization of $\Farb_{N}(n)$ when $N$ is an infinite, finitely generated nilpotent group as seen in the following theorem. Note that the numbering and any unexplained terminology comes from \cite{Pengitore_1}. \newline

\begin{fake1}\label{bad_first_thm}
	Let $N$ be an infinite, finitely generated nilpotent group. Then there exists a $\psi_{\RD}(N) \in \N$ such that
	$
	\Farb_N(n) \approx \pr{\log(n)}^{\psi_{\RD}(N)}.
	$
	Additionally, we may compute $\psi_{\RD}(N)$ given a basis for $\ga_c(N / T(N))$ where $c$ is the step length of $N / T(N)$. \newline
\end{fake1}

Khalid Bou-Rabee communicated to us a counterexample to \cite[Theorem 1.1]{Pengitore_1}. To be specific, he constructed a torsion free, finitely generated nilpotent group $N$ where \cite[Theorem 1.1]{Pengitore_1} predicts that $\Farb_{N}(n)$ grows asymptotically as $\pr{\log(n)}^5$, but where it can be shown that $\Farb_N(n)$ is bounded asymptotically above by $\pr{\log(n)}^4$.   Upon inspection of the article, it turns out that \cite[Proposition 4.10]{Pengitore_1} is incorrect (see Proposition \ref{incorrect_prop}). The upper bound for $\Farb_{N}(n)$ produced in \cite{Pengitore_1}, while no longer sharp, is still correct. The purpose of this article is to provide asymptotic bounds for $\Farb_N(n)$ when $N$ is finitely generated nilpotent group which takes into account Bou-Rabee's construction. We start by giving an exact computation of $\Farb_N(n)$ when the nilpotent group has an abelian, torsion free quotient, and when the nilpotent group has a non-abelian, torsion free quotient, we provide a lower asympototic bound for $\Farb_N(n)$ in terms of the step length of $N$. Moreover, we provide conditions for when we can construct an upper asymptotic bound for $\Farb_{N}(n)$ that is better than the one given by \cite[Theorem 1.1]{Pengitore_1}. Finally, we provide conditions on the nilpotent group for when our methods can explicitly compute $\Farb_N(n)$. 

Before we start, we introduce some notation. For two nondecreasing functions $\map{f,g}{\N}{\N}$, we say that $f \precsim g$ if there exists a constant $C>0$ such that $f(n) \leq C \: g(Cn)$. We say that $f \approx g$ if $f \precsim g$ and $g \precsim f$. For a group $G$, we denote $\ga_i(G)$ as the $i$-th step of the lower central series of $G$. Finally, for a finitely generated nilpotent group $N$, we denote $h(N)$ as its Hirsch length and define $T(N)$ as the subgroup generated by finite order elements (see Section 2 for the definitions of $h(N)$ and $T(N)$).

\subsection{Residual finiteness}
We start with the case of infinite, finitely generated nilpotent groups $N$ where $N/T(N)$ is abelian.
\begin{thm}\label{torsion_free_abelian}
	Let $N$ be an infinite, finitely generated nilpotent group such that $N/T(N)$ is abelian. Then
	$$
	\Farb_{N}(n) \approx \log(n).
	$$
\end{thm}

The situation is more interesting when $N/T(N)$ has step length $c > 1$ as seen in the following theorem.
\begin{thm}\label{lower_bound_rf_result}
		Let $N$ be an infinite, finitely generated nilpotent group such that $N / T(N)$ has step length $c > 1$. There exists a natural number $\dim_{\RFL}(N)$ such that $\dim_{\RFL}(N) \geq c + 1$ and where
		$$
		\pr{\log(n)}^{\dim_{\RFL}(N)} \precsim \Farb_{N}(n).
		$$
	\end{thm}
	
	For this theorem, we need to find an infinite sequence of elements $\set{x_i}_{i=1}^{\infty}$ such that the minimal finite order group $Q_i$ where there exists a surjective group morphism $\map{\varphi_i}{N}{Q_i}$ satisfying $\varphi_i(x_i) \neq 1$ has order approximately $\pr{\log(\|x_i\|}^{\dim_{\RFL}(N)}$. We start the search for this sequence by demonstrating that we may assume that the nilpotent group $N$ is torsion free. We then find an element $x \in N$ such that there exists an infinite sequence of natural numbers $\set{m_i}_{i=1}^{\infty}$ for which our desired sequence of elements is given by $\set{x^{m_i}}_{i=1}^{\infty}$.  Any such element $x$ will be a primitive element in the isolator of $\ga_c(N)$, and associated to these elements, there exists a natural number invariant $\dim_{\RFL}(N,x)$ which captures the lower asymptotic behavior of separating powers of $x$ from the identity with surjective group morphisms to finite $p$-groups as one varies the prime number $p$. Letting $\RP_{N,x,i}$ be the set of primes $p$ such the order of a minimal finite $p$-group which witnesses the primitive element $x$ has order $p^i$, we take $\dim_{\RFL}(N,x)$ to be the minimal $i_0$ such that $|\RP_{N,x, {i_0}}| = \infty$. By maximizing $\dim_{\RFL}(N,x)$ over all primitive elements in the isolator of $\ga_{c}(N)$, we obtain the value $\dim_{\RFL}(N)$.
	
	For the next theorem, we say a finitely generated nilpotent group has \textbf{tame residual dimension} if the sets $\RP_{N,x,i}$ have a natural density for all primitive elements $x$ in the isolator of $\ga_c(N)$ and indices $i$ and if there exists a generating basis $\{z_i\}_{i=1}^h$ for the isolator of $\ga_c(N)$ such that for any primitive element $g$ and prime $p$ there exists a surjective group morphism from $N$ to a $p$-group of minimal order that witnesses both $g$ and $z_i$ for some index $i$.  This definition implies that we may associate a natural number invariant to any such primitive element $x$ which takes into account the smallest index $i$ where we may apply the Prime Number Theorem relative to the set $\RP_{N,x,i}$. Moreover, we have that the asymptotic behavior of separating powers of any primitive element from the identity using surjective group morphisms to finite $p$-groups as we vary the prime $p$ is controlled by the asumptotic behavior of separating powers of elements of the given generating basis from the identity using surjective group morphism to finite $p$-groups as we vary the prime $p$. We denote this value as $\dim_{\RFU}(N,x)$, and note that our assumption implies that 
$$
\text{min}\{\dim_{\RFU}(N,z_i) \: | \: 1 \leq i \leq h\} \leq \dim_{\RFU}(N,x) \leq \text{max}\{\dim_{\RFU}(N,z_i) \: | \: 1 \leq i \leq h\}.
$$
By maximizing $\dim_{\RFU}(N,z_i)$ over elements of the generating basis, we obtain the value $\dim_{\RFU}(N)$ which gives a degree for a polynomial in logarithm upper asymptotic bound for $\Farb_N(n)$ for this class of nilpotent groups. As Bou-Rabee's example demonstrates, there exists a class of nilpotent groups who residually finiteness is strictly slower than what is predicted by \cite[Theorem 1.1]{Pengitore_1}, and the motivation in introducing the definition of tame residual dimension is to precisely capture this phenomenom. Thus, we have the following theorem.
	\begin{thm}\label{upper_bounds_residual_function}
		Let $N$ be an infinite, finitely generated nilpotent group. Then
		$$
		\Farb_N(n) \precsim \pr{\log(n)}^{\psi_{\RD}(N)}.
		$$ Now suppose that $N$ has tame residual dimension. Then there exists a natural number $\dim_{\RFU}(N)$ satisfying $\dim_{\RFU}(N) \leq \psi_{\RD}(N)$ such that
		$$
		\Farb_N(n) \precsim \pr{ \log( n)}^{\dim_{\RFU}(N)}.
		$$ If
		$
		\dim_{\RFU}(N) < \psi_{\RD}(N),
		$
		then $\Farb_{N}(n)$ grows strictly slower than what is predicted by \cite[Theorem 1.1]{Pengitore_1}.
	\end{thm}
	We observe that the first statement of the above theorem is the upper bound produced for $\Farb_{N}(n)$ in \cite[Theorem 1.1]{Pengitore_1}, and we claim no originality in including it in this article. We restate this theorem and give its proof to show the similarities and differences between the proof of the original upper bound and the proof of improved upper bound in the case of finitely generated nilpotent groups that have tame residual dimension. In particular, we want to highlight how the geometry of finite $p$-quotients of a finitely generated nilpotent group changes as we vary the prime and how it contrasts with torsion free quotients of the nilpotent group with one dimensional center.

For this last class of nilpotent groups, we are able to apply our methods to completely characterize the asymptotic behavior of $\Farb_{N}(n)$. In addition to assuming that the sets $\RP_{N,x,i}$ have a well defined natural density for all primitive elements $x$ and indices $i$, we assume that these sets $\RP_{N,x,i}$ are finite when $i < \dim_{\RFU}(N,x)$. In particular, we may ignore these sets of primes when calculating an asymptotic lower bound for $\Farb_{N}(n)$. Finitely generated nilpotent groups that satisfy these properties are said to have \textbf{accessible residual dimension}, and as a consequence, we have that $\dim_{\RFL}(N) = \dim_{\RFU}(N)$ whose common value is denoted $\dim_{\ARF}(N)$.
	
	\begin{thm}\label{accessible_result}
		Let $N$ be an infinite, finitely generated nilpotent group such that $N / T(N)$ has step length $c > 1$, and suppose that $N$ has accessible residual dimension. Then there exists a natural number $\dim_{\ARF}(N)$ such that $c + 1 \leq \dim_{\ARF}(N) \leq \psi_{\RD}(N)$ and where
		$$
		\Farb_{N}(n) \approx \pr{\log(n)}^{\dim_{\ARF}(N)}.
		$$
	\end{thm}

\subsection{Previous bounds found in the literature}
To the author's knowledge, the best previously known asymptotic lower bound for $\Farb_N(n)$ for the class of infinite, finitely generated nilpotent groups was $\log(n)$ which was given by \cite{Bou_Rabee10}.  This lower bound does not use any of the structure of the nilpotent group other than it containing an infinite order element, and subsequently, the lower bound is not connected to the geometry of the nilpotent group in consideration. On the other hand, the lower bounds produced in this article reflect the step length of the nilpotent group and take into account the geometry of the pro-$p$ completion as the prime $p$ varies. In particular, this asymptotic lower bound represents a signicant improvement in the understanding of $\Farb_N(n)$.

The first asymptotic upper bound for $\Farb_N(n)$ in the literature was $(\log(n))^{h(N)}$ and was given by \cite[Theorem 0.2]{Bou_Rabee10}. This upper bound partially reflects the structure of $N$ in that the $h(N)$ gives the dimension of the real Mal'tsev completion of $N$. However, the bound $(\log(n))^{\psi_{\text{RF}}(N)}$ reflects the geometry of the different quotients of the real Mal'tsev completion which has one dimensional center in which a given primitive central element of $N$ does not vanish. In particular, the original bound given in \cite{Bou_Rabee10} can be improved as soon as the center of the Mal'tsev completion has dimension greater than $1$ which can be seen in the above theorem.

When the nilpotent group has tame residual dimension, the upper bounds can be improved even further. This improved bound reflects the number theoretic properties of these nilpotent groups in relation to how finite $p$-quotients behave as the prime $p$ varies. In particular, these finite $p$-quotients are related to the distribution of primes $p$ for which a polynomial $F(T)$ with integer coefficients has a solution mod $p$ as we will see in Bou-Rabee's example. As a consequence, the upper bound for residual finiteness for nilpotent groups with tame residual dimension reflects more strongly how geometry of the pro-$p$ completion of the given nilpotent group varies with the prime $p$ and how their geometry differs from the Mal'tsev completion than what is seen in the literature or in the first asymptotic upper bound given in this article.

\subsection{Partial recovery of the asymptotic behavior of $\Farb_N(n)$}
It is important to note that we are unable to fully recover the asymptotic behavior of $\Farb_{N}(n)$ for a general finitely generated nilpotent group $N$; however, if we could show that the sets $\RP_{N,x,i}$ have positive natural density or are finite, we would be able to fully characterize $\Farb_{N}(n)$. For reasons related to the structure of nilpotent linear algebraic groups defined over $\Q$, we believe this to be the case and go into detail in the next few paragraphs.

Up to passing to a finite index subgroup, every torsion free, finitely generated nilpotent group $N$ arises as the group of integral points of a connected linear algebraic group $\textbf{N}$ defined over $\Q$ where the image of $N$ under the inverse of the exponential map provides a $\Z$-structure for the Lie algebra $\Fr{n}$ of $\textbf{N}$. Since $\textbf{N}$ satisfies strong approximation, we have that the smooth mod $p$ reduction of $\textbf{N}$, denoted $\textbf{N}_p$,  is well defined for all but finitely many primes $p$, and for such primes $p$, we have that the image of $N$ under the natural reduction group morphism is isomorphic to $\textbf{N}_p(\mathbb{F}_p)$ where $\textbf{N}_p(\mathbb{F}_p)$ is the group of $\mathbb{F}_p$-points of $\textbf{N}_p$. In particular, the Lie algebra of $\textbf{N}_p$, denoted as $\Fr{n}_p$ is given by the mod $p$ reduction of the $\Z$-structure on $\Fr{n}$. Thus, if there exists a normal connected algebraic subgroup $\textbf{H}_p \nsub \textbf{N}_p$ such that $(\textbf{N}_p / \textbf{H}_p)(\mathbb{F}_p)$ witnesses $x$ and where $\dim_{\mathbb{F}_p}(\textbf{N}_p/ \textbf{H}_p) = i$ for all primes $p \in \RP_{N,x,i}$, we may reduce the study of separating powers of $x$ to the study of quotients of $\Fr{n}_p$ in which a nonzero vector $X \in \Fr{n}$ does not vanish as we vary the prime $p$. At this point, we would use Lie theoretic methods and basic algebraic geometry as a way to study Lie quotients of $\Fr{n}_p$ for different primes $p$. This perspective gives more traction to study the sets $\RP_{N,x,i}$ and provides intuitive justification for the sets $\RP_{N,x,i}$ having a natural density.

For $p \in \RP_{N,x,i}$ where $i <\dim_{\RFU}(N,x)$, we have that any finite $p$-group $P$ that witnesses $x$ and where $|P| = p^i$ may have step length less than that of $N$. These primes would correspond to  finite quotients of the $\Z$-structure on $\Fr{n}$ which lose the structure of $N$. The idea is that the structure constants of the $\Z$-structure may vanish mod $p$; however, since there are finitely many structure constants, we would have that the Lie algebra $\Fr{n}_p$ would receive the full Lie bracket structure of $\Fr{n}$ for all but finitely many primes $p$. That should imply that the set $\RP_{N,x,i}$ is finite when $i < \dim_{\RFU}(N,x)$ for any primitive element $x$.

	\subsection{Plan of paper}
	In \S \ref{background}, we introduce necessary background and conventions for this paper.	\S \ref{motivation} expounds on the example provided to by Khalid Bou-Rabee. \S \ref{residual_set_up} introduces and defines the constants $\dim_{\RFL}(N),$ $ \dim_{\RFU}(N)$, and $\dim_{\ARF}(N)$ for infinite, finitely generated nilpotent groups. \S \ref{lower_bound_section}, \S  \ref{upper_bounds_section}, and \S\ref{accessible_section} provide proofs of Theorem \ref{lower_bound_rf_result}, Theorem \ref{upper_bounds_residual_function}, and Theorem \ref{accessible_result}, respectively.  \S \ref{separability_heisenberg} finishes with the computation of $\Farb_{N}(n)$ when $N$ is a Filiform nilpotent group.
	
	\subsection{Acknowledgments}
	I want to thank for my advisor Ben McReynolds for all of his help and guidance. I would like to thank Khalid Bou-Rabee for making me aware of his counterexample to \cite[Theorem 1.1]{Pengitore_1}. Finally, I would like to thank Daniel Groves for discussions regarding finite $p$-groups.
	\section{Background}\label{background}
	
	\subsection{Notation and conventions}
	We let $\lcm\set{r_1,\cdots,r_m}$ be the lowest common multiple of $\set{r_1,\cdots,r_m} \subseteq \Z$ with the convention that $\lcm(a) = |a|$ and $\lcm(a,0) = 0$. We let $\gcd(r_1,r_2)$ be the great common divisor of the integers $r_1$ and $r_2$ with the convention that $\gcd(r,0) = |r|$. We write $\mathbb{N}$ to be the set of natural numbers excluding $0$ and denote $\mathbb{P}$ as the set of prime numbers.
	
	We denote $\|g\|_S$ to be the word length of $g$ in $G$ with respect to the finite generating subset $S$, and when the subset $S$ is clear from context, we will write $\|g\|$. We denote the identity of $G$ as $1$ and denote $\set{1}$ as the trivial group.  We let $\text{Ord}_G(x)$ be the order of $x$ as an element of $G$ and denote the cardinality of a finite group $G$ as $|G|$. For a normal subgroup $H \nsub G$, we set $\map{\pi_{H}}{G}{G/ H}$ to be the natural projection and write $\bar{x} = \pi_{H}(x)$ when $H$ is clear from context. When given a nonempty subset $X \subset G$, we denote $\innp{X}$ as the subgroup generated by the elements of $X$. We denote $[x,y] = x^{-1} \: y^{-1} \: x \:y$. For $x_1,\cdots,x_k \in G$, we define $[x_1,\cdots,x_k] = [x_1,[x_2,\cdots,x_k]]$ where $[x_2,\cdots,x_k]$ is inductively defined. For nonempty subsets $A,B \subset G$, we define $[A,B]$ as the subgroup generated by the subset
	$
	\set{[a,b] \: | \: a \in A, b \in B}.
	$
	
	We denote the center of $G$ as $Z(G)$ and denote $\ga_i(G)$ as the $i$-th term of the lower central series of $G$. For a finitely generated nilpotent group $N$, we denote $h(N)$ as its Hirsch length and $T(N)$ as the subgroup generated by finite order elements. For natural numbers $m$, we let $G^m = \innp{g^m \: | \: g \in G}$ and write $\pi_m$ for the natural projection $\map{\pi_m}{G}{G / G^m}$.
	\subsection{Residual finiteness}
	Following \cite{Bou_Rabee10}, we define the depth function $\map{\D_{G}}{G \backslash \{1\}}{\N \cup \set{\infty}}$ of the finitely generated group $G$ as
	$$
	\D_{G}(g) = \text{min} \set{|Q| \: | \: \map{\varphi}{G}{Q}, \: |Q| < \infty, \text{ and } \varphi(g) \neq 1}
	$$
	with the understanding that $\D_{G}(g) = \infty$ if no such finite group exists. 
	\begin{defn}
		Let $G$ be a finitely generated group. We say that $G$ is \textbf{residually finite} if $\D_{G}(g) < \infty$ for all $g \in G \backslash \{1\}$.
	\end{defn}
	
	For a residually finite, finitely generated group $G$ with finite generating subset $S$, we define the associated complexity function $\map{\Farb_{G,S}}{\N}{\N}$ as
	$$
	\Farb_{G,S}(n) = \text{max}\set{\D_{G}(g) \: | \: g \in G \backslash \{1\} \: \text{ and } \|g\|_S \leq n}.
	$$
	For any residually finite, finitely generated group $G$ with finite generating subsets $S_1$ and $S_2$, we have that
	$
	\Farb_{G,S_1}(n) \approx \Farb_{G,S_2}(n)
	$
	(see \cite[Lemma 1.1]{Bou_Rabee10}). Henceforth, we will suppress the choice of finite generating subset.

	\subsection{Nilpotent groups} For more details of the following discussion, see \cite{Hall_notes,Segal_book_polycyclic}. 	We define $\ga_1(G) = G$ and inductively define $\ga_i(G) = [\ga_{i-1}(G),G]$. We call the subgroup $\ga_i(G)$ the \textbf{$i$-th term of the lower central series of $G$.}
	\begin{defn}
		Let $G$ be a finitely generated group. We say that $G$ is a \textbf{nilpotent group} if $\ga_k(G) = \set{1}$ for some natural number $k$. We say that $G$ is a \textbf{nilpotent group of step length $c$} if $c$ is the smallest natural number such that $\ga_{c+1}(G) = \set{1}$. If the step size is unspecified, we simply say that $G$ is a nilpotent group.
	\end{defn}

For a subgroup $H \leq N$ of a nilpotent group, we define the \textbf{isolator of $H$ in $N$} as
$$
\sqrt[N]{H} = \{ g \in N \:  \: | \exists \: \: k \in \N \:  \text{ such that }  \: g^k \in H \}.
$$ 
$\sqrt[N]{H}$ is a subgroup of $N$ for all $H \leq N$ when $N$ is a torsion free, finitely generated nilpotent group. Additionally, if $H \nsub N$, then $\sqrt[N]{H} \nsub N$. As a natural consequence, $N / \sqrt[N]{H}$ is torsion free. When the group $N$ is clear from context, we will simply write $\sqrt{H}$.

The \textbf{torsion subgroup} of a finitely generated nilpotent group $N$  is defined as
$$
T(N) = \set{g \in N \: | \: \text{Ord}_N(g) < \infty}.
$$
It is well known that $T(N)$ is a finite characteristic subgroup of $N$. Moreover, if $N$ is an infinite, finitely generated nilpotent group, then $N  / T(N)$ is a torsion free, finitely generated nilpotent group.

When given a torsion free, finitely generated nilpotent group, we may  refine the series $\{\sqrt{\ga_i(N)}\}_{i=1}^{c}$ to obtain a normal series $\set{H_i}_{i=1}^{h(N)}$ where $H_i / H_{i - 1} \cong \Z$ for all $i$. The number of terms in this series is known as the \textbf{Hirsch length} of $N$ and is denoted as $h(N)$. In particular, the Hirsch length may be computed as $$h(N) = \sum_{i=1}^{c}\text{rank}_\Z(\ga_i(N)/\ga_{i+1}(N)).$$ We choose $x_1 \in N$ such that $H_1 \cong \innp{x_1}$, and for each $2 \leq i \leq h(N)$, we choose $x_i \in H_i$ such that $H_i/ H_{i-1} \cong \innp{\pi_{H_{i-1}}(x_i)}$. Any generating subset chosen in this way will be called a \textbf{Mal'tsev basis}. Via \cite[Lemma 8.3]{Eick_Holt_obrien}, we may uniquely represent each element of $N$ with respect to this generating subset as
$
g = \prod_{i=1}^{h(N)} x_i^{\al_i}
$
where $\al_i \in \Z$. The values $\set{\al_i}_{i=1}^{h(N)}$ are referred to as the \textbf{Mal'tsev coordinates of $g$.} Whenever we reference a Mal'tsev basis, we suppress reference to the series $\set{H_i}_{i=1}^{h(N)}$.

	The following proposition and its proof can be originally found in \cite[Lemma 3.8]{Pengitore_1}. It relates the Mal'tsev coordinates of an element $g$ to its word length with respect to the generating subset given by the Mal'tsev basis.
\begin{prop}\label{coord_bound}
	Let $N$ be a torsion free, finitely generated nilpotent group of step length $c$ with a Mal'tsev basis $\set{x_i}_{i=1}^{h(N)}$. Suppose that $g = \prod_{i=1}^{h(N)}x_i^{\al_i}$ is a nontrivial element where  $\|g\| \leq n$. For each $1 \leq i \leq c$, we define $M_i = \sqrt{\ga_i(N)}$, and for each $1\leq i \leq h(N)$, we define $t_i$ as the minimal natural number where $\pi_{M_{t_i}}(x_i) = 1$ and $\pi_{M_{t_i+1}}(x_i) \neq 1$. Then $|\al_i| \leq C \: n^{t_i}$ for all $i$ where $C > 0$ is some constant.
\end{prop}
\begin{proof}
	We proceed by induction on step length, and since the base case of abelian groups is evident, we may assume that $N$ has step length $c > 1$. We observe that the image of the set $\set{x_i}_{i=h(M_c) + 1 }^{h(N)}$ in $N /M_c$ is a Mal'tsev basis for $N / M_c$ and that $\pi_{M_c}(M_t) = \sqrt{\ga_{t}(N/M_c)}$. Therefore, the inductive hypothesis implies that there exists a constant $C_1 > 0$ such that if $h(M_c) + 1 \leq i \leq h(N)$, then $|\al_i| \leq C_1 \: n^{t_i}$. For each $1 \leq i \leq h(N)$, there exists a minimal natural number $\ell_i$ such that $x_i^{\ell_i} \in \ga_{t_i}(N)$. In particular, we may write $\al_i = s_i \ell_i + r_i$ where $0 \leq r_i < \ell_i$. We let $D = \text{max}\set{\ell_i \: | \: 1 \leq i \leq h(N)}$.
	
	To proceed, we will demonstrate that we may assume that $g \in M_c$. We have that
	$$
	|s_i \: \ell_i| \leq |s_i \: \ell_i + r_i - r_i| \leq |\al_i| + |r_i| \leq C_1 \: n^{t_i} + D \leq C_2 \: n^{t_i}
	$$
	for $h(M_c) + 1 \leq i \leq h(N)$ where $C_2 > 0$ is some constant. Thus, $|s_i| \leq C_2 \: n^{t_i}$ for all $h(M_c) + 1 \leq i \leq h(N)$. By \cite[3.B2]{Gromov}, we have that $\|x_i^{s_i \: \ell_i}\| \approx |s_i|^{1/t_i}$. Therefore, there exists a constant $C_{3,i} > 0$ such that $\|x_i^{s_i \: \ell_i}\| \leq C_{3,i} \: |s_i|^{1/t_i}$. In particular, we have that 
	$$
	\|x_i^{\al_i}\| = \|x_i^{s_i \: \ell_i + r_i}\| \leq \|x_i^{s_i \: \ell_i}\| + \|x_i^{r_i}\| \leq C_{3,i} |s_i|^{1/t_i} + D\leq D \: C_2^{1/t_i} \: C_{3,i} \: n.
	$$
	Hence, by setting $C_4 = \text{max}\{D \: C_2^{1/t_i} \: C_{3,i} \: | \: h(H_c) + 1 \leq i \leq h(N)\}$, we may write $\|x_i^{\al_i}\| \leq C_4 \: n$ for $h(M_c) + 1 \leq i \leq h(N)$. Letting $h = \prod_{i = h(M_c) + 1}^{h(N)} x_i^{\al_i}$, one can see that $g \: h^{-1} \in M_c$ and that 
$$\|g \: h^{-1}\| \leq \|g\| + \|h\| \leq \|g\| + \sum_{i=h(M_c)+1}^{h(N)} \|x_i^{\al_i}\| \leq \|g\| + h(N) \:C_4 \: n \leq C_5 \: n$$
	for some constant $C_5 > 0$ which gives our claim.
	
	By passing to the quotient $N/K_i$ where $K_i = \innp{x_j}_{j=1,j \neq i}^{h(H_c)}$, it is straightforward to see that $\|x_i^{\al_i}\| \leq C_5 \: n$ for each $1 \leq i \leq h(M_c)$. Using the same arguments as above, \cite[3.B2]{Gromov} implies that $|s_{i}|\leq C_{6,i} \: n^c$ for some constant $C_{6,i} > 0$ for each $1 \leq i \leq h(M_c)$. Thus, we may write
	$$
	|\al_{i}| = |s_{i} \ell_{i} + r_{i}| = |s_{i}||\ell_{i}| + |r_{i}| \leq D + D \: C_{6,i} \: n^c \leq (D+1) \: C_{6,i} \: n^c.
	$$
	Letting $C_7 = \text{max}\set{C_1,C_{6,1},\cdots,C_{6,h(H_c)}}$, we have by construction that $|\al_i| \leq C_7 \: n^{t_i}$ for all $i$.
\end{proof}

	\subsection{Density}
For $A \subset \mathbb{P}$, we define the \textbf{natural density of $A$ in $\mathbb{P}$} as 
$$
\delta(A) = \lim_{n \rightarrow \infty} \frac{|A \cap \set{1,\cdots,n}|}{|\mathbb{P}\cap \set{1,\cdots,n}|}
$$
when the limit exists. When this limit exists for a set $A$, we say that $A$ has a natural density.

\section{A counterexample to \cite[Proposition 4.10]{Pengitore_1} and \cite[Theorem 1.1]{Pengitore_1}}\label{motivation}
The following example was communicated to us by Khalid Bou-Rabee. Let $G$ be the torsion free, finitely generated nilpotent group given by
$$
G = \set{x,y,w,z,u,v \: | \: [x,y] =  [w,z] = 1, [x,w] = [y,z] = u, [x,z] = v, [y,w] = v^{-1}, u \text{ and } v \text{ are central}}.
$$

We start by listing some basic facts about $G$. The set $S = \set{x,y,w,z,u,v}$ is a Mal'tsev basis for $G$ from which it follows that $h(G) = 6$. Additionally, we have that the abelianization of $G$ is given by $G / \ga_2(G) \cong \set{\bar{x},\bar{y},\bar{w},\bar{z}}$ and that the center is given by $Z(G) \cong \innp{u,v}$. Finally, we observe that $G$ has step length $2$ and that $\ga_2(G) = Z(G)$.

The main tool used in \cite{Pengitore_1} is the following proposition.
We first introduce a definition.

\begin{defn} Let $N$ be a torsion free, finitely generated nilpotent group of step length $c$ with a Mal'tsev basis $\set{x_i}_{i=1}^{h(N)}$. For a vector $\vec{a} = \pr{a_i}_{i=1}^{\ell}$ where $a_i$ is a natural number such that $1 \leq a_i \leq h(N)$ for all $i$, we denote $[x_{\vec{a}}] = [x_{a_1}, \cdots, x_{a_{\ell}}]$.  We call $[x_{\vec{a}}]$ a simple commutator of weight $\ell$ with respect to $\vec{a}$, and since $N$ is a nilpotent group, we have that all simple commutators of weight greater than $c$ are trivial. We denote $W_k(N,\set{x_i})$ to be the set of nontrivial simple commutators in $\set{x_i}$ of weight $k$, and let $W(N,\set{x_i})$ to be the set of nontrivial commutators.
	
	We may write $[x_{\vec{a}}] = \prod_{i=1}^{h(N)}x_i^{\delta_{\vec{a},i}}.$
	Let
	$$
	\text{B}(N,\set{x_i}) = \lcm \set{|\delta_{\vec{a},i} | \: | 1 \leq i \leq h(N), \delta_{\vec{a},i} \neq 0, \text{ and } [x_{\vec{a}}] \in W(N, \set{x_i})}.
	$$
\end{defn}

See \cite[Proposition 4.10]{Pengitore_1} for where the following proposition is originally found.
\begin{prop}\label{incorrect_prop}
	Let $N$ be a torsion free, finitely generated nilpotent group with a Mal'tsev basis $\set{x_i}_{i=1}^{h(N)}$. Let $\map{\varphi}{N}{Q}$ be a surjective group morphism to a finite $p$-group where $p > \text{B}(N,\set{x_i})$. Suppose that $\varphi([x_{\vec{a}}]) \neq 1$ for all $[x_{\vec{a}}] \in W(N,\set{x_i}) \cap Z(N)$. Also, suppose that $\varphi(x_i) \neq 1$ for $x_i \in Z(N)$ and $\varphi(x_i) \neq \varphi(x_j)$ for all $x_i,x_j \in Z(N)$ where $i \neq j$. Then $\varphi(x_t) \neq 1$ for $1 \leq t \leq h(N)$ and $\varphi(x_i) \neq \varphi(x_j)$ for $1 \leq i < j \leq h(N)$. Finally, $|Q| \geq p^{h(N)}$.
\end{prop}

For the group $G$ defined at the beginning of this subsection, we note that $W(G,S)\cap Z(G) = \set{u,v,u^{-1}}$ and that $B(G,S) = 1$. The following proposition produces a positive natural density subset of prime numbers $p$ where there exists a finite $p$-quotient $Q_p$ of $G$ such that the hypotheses of Proposition \ref{incorrect_prop} are satisfied and where $|Q_p| = p^4$. Since Proposition \ref{incorrect_prop} predicts $|Q_p| = p^5$, we have an infinite collection of counterexamples for Proposition \ref{incorrect_prop}.

Before starting, we introduce some notation. Let $E = \set{p \in \mathbb{P} \: | \: 4 \text{ divides } p -1}$. For $p \in E$, we let $\set{a_{p},b_{p}}$ be the two distinct solutions to the equation $T^2 + 1 \equiv 0 \text{ mod } p$. Finally, we let $A_{p}$ and $B_{p}$ be the normal closures of the subgroups $\innp{x^{a_{p}} \: y}$ and $\innp{x^{b_{p}} \: y}$ in $G$, respectively.

\begin{prop}\label{incorrect_prop_counterexamples}
 If $p \in E$, then $\pi_{p}(A_{p}) \cap Z(G / G^{p}) \cong\Z / p \Z$ and $\pi_{p}(B_{p}) \cap Z(G / G^{p}) \cong \Z / p \Z.$ Additionally, $|G / G^{p} \cdot A_{p}| = |G / G^{p} \cdot B_{p} | = {p}^{4}$ and $Z(G / G^{p} \cdot A_{p}) \cong Z(G / G^{p} \cdot B_{p}) \cong \Z / p\Z.$ Moreover, we have that $\pi_{p}(A_{p}) \cap \pi_{p}(B_{p}) \cong \set{1}$ and $\innp{\pi_{p}(A_{p}), \pi_{p}(B_{p})} \cong Z(G / G^{p})$. Finally, we have that $\pi_{G^{p} \cdot A_{p}}(u),\pi_{G^{p} \cdot A_{p}}(v) \neq 1$, $\pi_{G^{p} \cdot B_{p}}(u), \pi_{G^{p} \cdot B_{p}}(v) \neq 1$, $
	\pi_{G^{p} \cdot A_{p}}(u) \neq \pi_{G^{p} \cdot A_{p}}(v)$, and that $\pi_{G^p \cdot B_{p}}(u) \neq \pi_{G^{p} \cdot B_{p}}(v).$

\end{prop}
\begin{proof}
	For the first statement, it is sufficient to prove that $|G / G^{p} \cdot A_{p}| = p^{4}$ and that $Z(G / G^{p}) \cap \pi(A_{p}) \cong \Z / p \Z$. By direct calculation, we have that $A_{p} \cong \innp{x^{a_{p}} \: y, u^{a_{p}} \: v^{-1}, u \: v^{a_{p}}}.$ Thus, $A_{p} \cap Z(G) \cong \innp{u^{a_{p}} \: v^{-1}, u \: v^{a_{p}}}.$
	Hence,
	$
	(u^{a_{p}} \: v^{-1})^{-a_{p}} = u^{-(a_{p})^2} \: v^{a_{p}} = u \: v^{a_{p}}  \text{ mod } G^{p}. 
	$
	Thus, 
	$
	\pi_{p}(A_{p}) \cap Z(G / G^{p}) \cong \innp{\pi_{p}(u \: v^{a_{p}})} \cong \Z / p \Z.
	$
	We note that each element $G / G^{p} \cdot A_{p}$ can be written uniquely as $\pi_{G^p \cdot A_p}(x^{\al_x} \: w^{\al_w} \: z^{\al_z} \: v^{\al_v})$ for natural numbers satisfying $0 \leq \al_x,\al_w,\al_z,\al_v < p$; thus, the second paragraph after \cite[Definition 8.2]{Eick_Holt_obrien} implies that $|G / G^{p} \cdot A_{p}| = p^{4}$. Moreover, we have that $\ga_2(G / G^{p} \cdot A_{p}) \cong Z(G / G^{p} \cdot A_{p})$. Hence, $Z(G / G^{p} \cdot A_{p}) \cong \Z / p \Z$.
	
	For the next statement, we note that $\pi_{p}(A_{p}) \cong \innp{u \: v^{a_{p}}}$ and that $\pi_{p}(B_{p}) \cong \innp{u \: v^{b_{p}}}$. Since $a_{p} \not \equiv b_{p} \: \text{ mod } p$, we have that $u \: v^{a_{p}} \neq u \: v^{b_{p}} \: \text{ mod } G^{p}$. Suppose for a contradiction that there exists a natural number $\ell$ such that $(u \: v^{a_{p}})^\ell	=   u \: v^{b_{p}} \text{ mod } G^{p}$. Since $(u \: v^{a_{p}})^\ell = u^\ell \: v^{\ell \: a_{p}}$, we must have that $\ell \equiv 1 \text{ mod } p$ and $\ell \: a_{p} \equiv  b_{p} \text{ mod } p$. Since $\ell \: a_{p} \equiv a_{p} \text{ mod } p$, we have that $a_{p} \equiv b_{p} \text{ mod } p$ which is a contradiction. In particular, $\pi_{p}(A_{p}) \cap \pi_{p}(B_{p}) = \set{1}$; hence, $\innp{\pi_{p}(A_{p}), \pi_{p}(B_{p})} \cong \Z / p \Z \times \Z / p \Z$. Since $Z(G / G^{p}) \cong \Z / p \Z \times \Z / p \Z$, it follows that $\innp{\pi_{p}(A_{p}), \pi_{p}(B_{p})} \cong Z(G / G^{p}).$
	
	The remaining two statements are evident.
\end{proof}
\cite[Theorem 1.1]{Pengitore_1} predicts that $\Farb_{G}(n) \approx \pr{\log(n)}^5$, and the following proposition provides a counterexample.
\begin{prop}\label{counter_example_first_thm_first_paper}
$
\Farb_{G}(n) \precsim \pr{\log(n)}^4.
$	
\end{prop}
\begin{proof}

	Let $g \in G$ be a nontrivial element such that $\|g\| \leq n$. We may write
	$
	g = x^{\al_x} \: y^{\al_y} \: w^{\al_w} \: z^{\al_w} \: u^{\al_u} \: v^{\al_v}. 
	$
	Proposition \ref{coord_bound} implies that there exists a constant $C_1 > 0$ such that $
	|\al_x|,|\al_y|,|\al_w|,|\al_z| \leq C_1 \: n
	$
	and that 
	$
	|\al_u|,|\al_v| \leq C_1 \: n^2.
	$ 
	
	Suppose that $\pi_{\ga_2(G)}(g) \neq 1$. \cite[Corollary 2.3]{Bou_Rabee10} implies that there exists a surjective group morphism $\map{\varphi}{G/ \ga_2(G)}{P}$ to a finite group where $|P| \leq C_1 \: C_2 \: \log(C_1 \: C_2 \: n)$ for some constant $C_2 > 0$ such that  $\varphi(\pi_{\ga_2(G)}(g)) \neq 1$. Therefore, $\D_G(g) \leq C_1 \: C_2 \: \log(C_1 \: C_2 \: n).$
	
	Now suppose that $\pi_{\ga_2(G)}(g) = 1$. That implies that we may write $g = u^{\al_u} \: v^{\al_v}$. Let $E$, $A_p$, and $B_p$ be defined as above.  Without loss of generality, we may assume that $\al_u \neq 0$. Since Chebotarev's Density Theorem implies that $\delta(E) > 0$, the Prime Number Theorem \cite[1.2]{Tenenbaum} implies that there exists a prime number $p \in E$ such that $p \nmid \al_u$ and where $p \leq C_3 \: \log(C_3 \: |\al_u|)$ for some constant $C_3 > 0$. Therefore, there exists a constant $C_4 > 0$ such that $p \leq C_4 \: \log(C_4 \: n)$. Thus, $\pi_p(g) \neq 1$. Proposition \ref{incorrect_prop_counterexamples} implies that $\pi(A_p) \cap Z(G / G^p) \cong \Z / p \Z$ and that $\pi(N_p) \cap Z(G / G^p) \cong \Z / p \Z$. Since $Z(G / G^p) \cong \Z / p \Z \times \Z / p \Z$, we may assume that $\pi_p(g)\notin \pi(A_p)$. Thus, $\pi_{G^p \cdot A_p}(g) \neq 1$, and Proposition \ref{incorrect_prop_counterexamples} implies that $|G / G^p \cdot A_p| = p^4$. Hence, there exists a constant $C_5 > 0$ such that
	$
	\D_{G}(g) \leq C_5 \: \pr{\log(C_5 \: n)}^4.
	$
	Hence,
	$
	\Farb_{G}(n) \precsim \pr{\log(n)}^4.
	$
\end{proof}

	\section{Residual dimension}\label{residual_set_up}
The purpose of this section is to define the constants $\dim_{\RFL}(N)$, $\dim_{\RFU}(N)$, and $\dim_{\ARF}(N)$ for a torsion free, finitely generated nilpotent group $N$. We start by giving a lower bound for the order of a finite $p$-group in terms of the prime $p$ and its step length.
\begin{lemma}\label{p_dim_step_length}
	If $Q$ is an abelian finite $p$-group, then $|Q| \geq p$. If $Q$ has step length $c > 1$, then $|Q| \geq p^{c + 1}$.
\end{lemma}
\begin{proof}
	Since the first statement is clear, we may assume that $Q$ has step length $c > 1$. We prove the second statement by induction on step length. For the base case, we assume that $Q$ has step length $2$. There exist elements $x,y \in Q$ such that $[x,y] \neq 1$. Since $Q$ has step length $2$, we have that $[x,y] \in [Q,Q] \leq Z(Q)$. Consider the group $H = \innp{x,y,[x,y]}$. Since each element in $H$ can be written uniquely as $x^t \: y^s \: [x,y]^\ell$ for integers satisfying $0 \leq t < \text{Ord}_Q(x)$, $0 \leq s < \text{Ord}_Q(y)$, and $0 \leq \ell < \text{Ord}_Q([x,y])$, we observe that the second paragraph after \cite[Definition 8.2]{Eick_Holt_obrien} implies that $|H| = \text{Ord}_Q(x) \cdot \text{Ord}_Q(y) \cdot \text{Ord}_Q([x,y])  \geq p^3.$ Since $|H|$ divides $|Q|$, we have that $|Q| \geq p^3$.
	
	Now suppose that $Q$ has step length $c > 2$. By induction, $|Q / \ga_c(Q)| \geq p^c$, and since $\ga_c(Q)$ is abelian, we have that $|\ga_c(Q)| \geq p$. In particular, $|Q| = |Q / \ga_c(Q)||\ga_c(Q)| \geq p^{c + 1}$.
\end{proof}

We say that an infinite order element $g$ of a torsion free finitely generated group $N$ is \textbf{primitive} if whenever there exists an element $h \in N$ and a non-zero integer $m$ such that $h^m = g$, then either $g = h$ or $g = h^{-1}$. In particular, if $N$ is a torsion free, finitely generated abelian group with a primitive element $z$, there exists a Mal'tsev generating basis $\set{x_i}_{i=1}^{h(N)}$ for $N$ as a $\Z$-module such that $x_1 = z$. One way to see that is to first note that $N / \innp{z} \cong \Z^{h(N) - 1}$ since $z$ is not the proper power of an element. By lifting a Mal'tsev generating basis for $\Z^{h(N) - 1}$, we obtain a Mal'tsev generating basis for $N$ which includes $z$.  The following lemma implies for any prime number $p$ that we may separate a primitive central element $x$ in a torsion free, finitely generated nilpotent group $N$ from the identity with a surjective group morphism to a finite $p$-group. 

\begin{lemma}\label{f_p_witness}
	Let $N$ be a torsion free, finitely generated nilpotent group, and let $z \in Z(N)$ be a primitive element. Let $p$ be a prime number. There exists a surjective group morphism $\map{\varphi}{N}{Q}$ to a finite $p$-group $Q$ such that $\varphi(z) \neq 1$ and where $|Q| \leq p^{h(N)}$.
\end{lemma}
\begin{proof}Let $\set{x_i}_{i=1}^{h(N)}$ be a Mal'tsev basis for $N$. We may write $z = \prod_{i=1}^{h(N)} x_i^{\al_i},$ and since $z$ is a primitive element, there exists an index $i_0$ such that $\al_{i_0} \not \equiv 0 \text{ mod } p$.  Since $\pi_p(Z(N)) \cong \prod_{i=1}^{h(Z(N))}\Z / p \Z$, we have that each element of $\pi_Z(N/N^p)$ may be written uniquely as $\pi_p(\prod_{i=1}^{h(Z(N))}x_i^{\beta_i})$ where $0 \leq \beta_i < p$. Thus, we have that $\pi_p(z) \neq 1$. One last observation is that $|N / N^p| = p^{h(N)}$ as desired.
\end{proof}
The above lemma implies that we are able to find a surjective group morphism from $N$ to a finite $p$-group of minimal order where the image of $z$ is not trivial. Thus, we have the following definition.
\begin{defn}\label{residual_p_dimension}
	Let $N$ be a torsion free, finitely generated nilpotent group with a prime number $p$ with a primitve element $z \in Z(N)$ be a primitive element. Proposition \ref{f_p_witness} implies that there exists a surjective group morphism $\map{\varphi}{N}{P}$ to a finite $p$-group such that $\varphi(z) \neq 1$ and $|P| = p^k$ where 
	$$
	k = \text{min}\set{m \: | \:  \exists \: \map{\varphi}{N}{Q} \: \text{ that satisfies Proposition \ref{f_p_witness} for $z$ and where $|Q| = p^m$}}.
	$$
	We refer to $\map{\varphi}{N}{P}$ as a\ \textbf{$p$-witness of $z$}. 
\end{defn}

The next few statements establish some properties of $p$-witnesses.

\begin{lemma}\label{cyclic_p_witness}
Let $N$ be a torsion free, finitely generated nilpotent group with a primitive element $g \in \sqrt{\ga_c(N)}$.  If $\varphi: N \to Q$ is a $p$-witness of $N$ of $g$ where $p$ is some prime, then $\innp{\varphi(g)} \cong \Z / p \Z$. In particular, $Z(Q)$ is cyclic.
\end{lemma}
\begin{proof}

Suppose for a contradiction that  $Z(Q)$ is not cyclic. There exists a generating subset $\{x_i\}_{i=1}^{k}$ for $Z(Q)$ such that every element $h \in Z(Q)$ may be written $h = \prod_{i=1}^k x_i^{t_i}$ where $0 \leq t_i < p_i^{\al_i}$ for some $\al_i \geq 1$. There exist natural numbers $\{s_i\}_{i=1}^k$ such that $\varphi(g) = \sum_{i=1}^{k}x_i^{s_i}$ and where $s_j \neq 0$ for some $j$. Since $P = \innp{x_i \: | \: i \neq j, 1 \leq i \leq k}$ is a normal subgroup, we have that $\pi_P \circ \varphi : N \to Q/P$ satisfies $\pi_P \circ \varphi(g) \neq 1$ which contradicts the definition of a $p$-witness.

Thus,  if $\varphi: N \to Q$ is a $p$-witness of $g$, then $Z(Q)$ is cyclic, and since $\varphi(g) \neq 1$, we have that $\innp{\varphi(g)} \cong \Z / p^t  \Z$ for some $t$. If $t > 1$, then by letting $P = \innp{\varphi(g)^{p}}$, we have that $\varphi(g) \notin P$. Thus, we have that $\pi_P \circ \varphi(g) \neq 1$ and $|Q / P| < |Q|$ which contradicts the fact that $Q$ is a $p$-witness of $g$.
\end{proof}

The following proposition relates the existence of $p$-witness for an arbitrary primitive element with the $p$-witnesses of elements of a Mal'tsev generating basis for $\sqrt{\ga_c(N)}$ under certain assumptions on the generating basis. For notational simplicity, we let $H = \sqrt{\ga_c(N)}$ and $h = h(\sqrt{\ga_c(N)})$.

\begin{prop}\label{p_dimension_basis}
Let $N$ be a torsion free, finitely generated nilpotent group, and let $\{x_i\}_{i=1}^{h}$ be a generating basis for $H$. Let $g = \prod_{i=1}^hx_i^{\al_i}$ be a primitve element, and  let $p$ be prime. For each $1 \leq i \leq h$, let $\varphi_i: N \to Q_i$ be a $p$-witness of $x_i$, and let $\psi:N \to P$ be a $p$-witness for $g$. Suppose for some $j$ where $p \nmid \al_j$, we have that $\varphi_{j}(g) \neq 1$. Then
$$
\text{min}\{ |Q_i| \: | \:  1 \leq i \leq h \} \leq |P| \leq \text{max}\{|Q_i| \: |  \: 1 \leq i \leq h \}
$$
\end{prop}
\begin{proof}
Since $\varphi_j(g) \neq 1$, we have by definition that
$$
|P| \leq |Q_j| \leq \text{max}\{ |Q_i| \: | \:  1 \leq i \leq h \}
$$

Given that the other side of the inequality is clear when $\text{min}\{ |Q_i| \: | \:  1 \leq i \leq h \} = 1$, we may assume that $\text{min}\{ |Q_i| \: | \:  1 \leq i \leq h \} > 1$. Now suppose for a contradiction that $|P| < \text{min}\{|Q_i| \: | \: 1 \leq i \leq h \}$. We then have by the definition of a $p$-witness that $\psi(x_i) = 1$ for each $1 \leq i \leq h$. Therefore, $\psi(g) = \prod_{i=1}^h \psi(x_i)^{\al_i} = 1$ which is a contradiction. Thus, $|P| \geq \text{min}\{|Q_i| \: | \: 1 \leq i \leq h \}.$
%For other side of the inequality, we claim that there exists an $i_0$ such that $\text{Ord}_P(\psi(x_{i_0})) \nmid \al_{i_0}$, and for a contradiction, suppose otherwise. Letting $\al_i = \text{Ord}_{P}(\psi(x_i)) \: m_i$ when $\text{Ord}_{P}(\psi(x_i)) \mid \al_i$ and letting $I = \{1 \leq i \leq h \: | \: \psi(x_i) = 1\}$,  we may write
%$$
%\psi(g) = \prod_{i=1}^h \psi(x_i)^{\al_i} = \prod_{i \notin I} \psi(x_i)^{\al_i} = \prod_{i \notin I} \psi(x_i)^{\text{Ord}_P(\psi(x_i)) \: m_i}  = 1
%$$
%which is a contradiction. In particular, we have that
%$$
%|P| \geq \text{min}\{ |Q_i| \: | \:  1 \leq i \leq h \text{ s.t. } \text{Ord}_P(\psi(x_i)) \nmid \al_i \}.
%$$
%Thus, we are done.
\end{proof}

We note that even for $\Z^k$ with a Mal'tsev generating basis $\{z_i\}_{i=1}^k$ there exist $p$-witnesses for $z_1$ such that a given primitive element must vanish. For instance, consider the surjective group morphism $\varphi \colon \Z^2 \to \Z / p \Z$ given by $z_1^{\al_1} \: z_2^{\al_2} \to \al_1 + \al_2 \text{ mod } p$. It is easy to see that $\varphi \colon \Z^2 \to \Z / p \Z$ is a $p$-witness for $z_1$ and $z_2$; however, the primitive element $g = z_1 \: z_2^{p-1}$ satisfies $\varphi(g) = 1$. On the other hand, we also have the surjective group morphism $\psi \colon \Z^2 \to \Z / p \Z$ given by $z_1^{\al_1} z_2^{\al_2} \to \al_1 \text{ mod } p$ which is easily seen to be a $p$-witness for $z_1$ satisfying $\psi(g) \neq 1$. This discussion demonstrates that not every finite $p$-witness for elements of a Mal'tsev generating basis for $\sqrt{\ga_c(N)}$ is useful for separating powers of an arbitrary primitive element in that there may be finite $p$-witness in which a fixed primitve element must vanish. That, in turn, implies that all powers of this element must also vanish. Thus, we have an interest in the existence of generating basis for $\sqrt{\ga_c(N)}$ which controls the order of a $p$-witness for any primitive element. Thus, we have the following definition.
\begin{defn}
Let $N$ be a torsion free, finitely generated nilpotent group of step length $c$ with a generating basis $\{x_i\}_{i=1}^h$ for $\sqrt{\ga_c(N)}$. We say that $\{x_i\}_{i=1}^h$ is a \textbf{residually tame basis} for $\sqrt{\ga_c(N)}$ if the following holds. Given any prime $p$ and any primitive element $g = \prod_{i=1}^{h}x_i^{\al_i}$, there exists an index $i$ such that $p \nmid \al_i$ and where there exists a finite $p$-witness $\varphi \colon N \to P$ for $x_i$ such that $\varphi(g) \neq 1$.
\end{defn}

In the following discussion, we will see that any torsion free, finitely generated abelian group has a residually tame generating basis. Note that any primitive element in $\Z^h$ can be written as $g = (\al_1, \cdots, \al_h)$ where $\gcd(\al_1, \cdots, \al_h) = 1$. Thus, if $p$ is a prime, there exists an $i_0$ such that $p \nmid \al_{i_0}$. Hence, we have that $\pi(g) \neq 0$ where $\pi \colon \Z^h \to \Z / p\Z$ be the projection onto the $i_0$-th coordinate mod $p$, and since $\pi \colon \Z^h \to \Z / p \Z$ is a $p$-witness for an element of our generating basis, we have our claim. Moreover, when $N$ is a torsion free, finitely generated nilpotent group where $h(\sqrt{\ga_c(N)}) = 1$, then by the above discussion we have that $\sqrt{\ga_c(N)}$ admits a residually generating basis. Thus, a natural question is whether a residually tame generating basis always exists for $\sqrt{\ga_c(N)}$ for any torsion free, finitely generated nilpotent group $N$. However, as of this writing, the author is unaware of such a construction of such a basis. As we will see later, the existence of a residually tame generating basis will be essential in producing upper asymptotic bounds for residual finiteness better than the ones found in \cite{Pengitore_1}.

\subsection{Lower residual dimension}
We start this subsection by introducing the following definition which will be important for both the upper and lower asymptotic bounds for residual finiteness of finitely generated nilpotent groups.
\begin{defn}
For each $1 \leq i \leq h(N)$, we define 
$$
\RP_{N,z,i} \bdef \{p \in \mathbb{P} \: | \:  \text{$p^i$ is the order of a $p$-witness of $z$}\}.
$$
We call $\RP_{N,z,i}$ the \textbf{set of residual prime numbers of $N$ with respect to $z$ of dimension $i$.}
\end{defn}

Suppose that $N$ is a torsion free, finitely generated nilpotent group of step length $c$ with a primitive element $z \in \sqrt{\ga_c(N)}$. Since the sets $\RP_{N,z,i}$ form a finite partition of the set of primes, we have by the Pidgeonhole principle that there  must be an index $1 \leq i_0 \leq h(N)$ such that $|\RP_{N,z,i_0}| = \infty$. That observation allows us to introduce a natural number that measures the lower asymptotic complexity of separating $z$ from the identity with surjective group morphisms to finite $p$-groups as we vary over all prime numbers $p$ based on the cardinality of $\RP_{N,z,i}$ for each $1 \leq i \leq h(N)$.

\begin{defn}
	Let $N$ be a torsion free, finitely generated nilpotent group of step length $c$ with a primitive element $z \in \sqrt{\ga_c(N)}$. There exists a minimal natural number $1 \leq t_0 \leq h(N)$ such that $|\RP_{N,z,t_0}| = \infty$. We call the natural number $t_0$ the \textbf{lower residual dimension of $z$ in $N$} and denote it as $\dim_{\RFL}(N,z)$. We call $\LR_{N,z} =\RP_{N,z,t_0}$ the \textbf{set of prime numbers that realize the lower residual dimension of $z$}.
\end{defn}
By maximizing over primitive elements in $\sqrt{\ga_{c}(N)}$, we obtain a natural invariant associated to any torsion free, finitely generated nilpotent group which gives the degree of polynomial in logarithm growth for a lower bound of residual finiteness.

\begin{defn}
	Let $N$ be a torsion free, finitely generated nilpotent group of step length $c$. Let 
	$$
	\dim_{\RFL}(N) \bdef \text{max}\{  \dim_{\RFL}(N,z) | \: z \in \sqrt{\ga_{c}(N)} \text{ is primitve}\}.
	$$
	We call $\dim_{\RFL}(N)$ the \textbf{lower residual dimension of $N$.} For an infinite, finitely generated nilpotent group $N$, we denote $\dim_{\RFL}(N) \bdef \dim_{\RFL}(N / T(N)).$ 
\end{defn}

We finish this subsection by giving a lower bound for $\dim_{\RFL}(N)$ in terms of the step length of $N$ when $N$ is a torsion free, finitely generated nilpotent group.
\begin{prop}\label{lower_bound_residual_dimension}
	Let $N$ be a torsion free, finitely generated nilpotent group of step length $c > 1$. Then $\dim_{\RFL}(N) \geq c+1$.
\end{prop}
\begin{proof}
Let $z \in \sqrt{\ga_c(N)}$ be a primitive element. Thus, there exists some natural number $k$ such that $z^k \in \ga_c(N)$.  Let $\map{\varphi}{N}{Q}$ be a $p$-witness $z$ where $p$ is a prime number that does not divide $k$. Since $p \nmid k$, we have that $\gcd(k,p) = 1$. In particular, $\innp{\varphi(z^k)} = \innp{\varphi(z)}$.
Since $\varphi(z^k) \neq 1$ and $z^k \in \ga_c(N)$,  we have that $Q$ has the same step length as $N$, and thus, Lemma \ref{p_dim_step_length} implies that $|Q| \geq c + 1$.

Let $A$ be the set of prime numbers that divide $k$, and let $B = \cup_{i = c + 1}^{h(N)}\RP_{N,z,i}$. We note that if $p \in \mathbb{P} \backslash A$, then the above claim implies that $p \in B$. Since $A$ is finite, we must have that $B$ is infinite and that $\RP_{N,z,t}$ is finite for $1 \leq t \leq c$. Thus, there exists a minimal index $i_0$ such that $i_0 \geq c + 1$ and where $\RP_{N,z,i_0}$ is infinite. That implies $\dim_{\RFL}(N,z) \geq c + 1$. By definition, $\dim_{\RFL}(N) \geq \dim_{\RFL}(N,z) \geq c + 1$. Hence, $\dim_{\RFL}(N) \geq c + 1$.
\end{proof}

\subsection{Upper residual dimension}
This subsection gives conditions on when the upper asymptotic bound for $\Farb_{N}(n)$ can be improved to be better than that of \cite[Theorem 1.1]{Pengitore_1}.

\begin{defn}\label{tame_residual_defn}
	Let $N$ be a torsion free, finitely generated nilpotent group of step length $c$, and suppose that $x \in \sqrt{\ga_{c}(N)}$ is a primitive element. Suppose for each $1 \leq i \leq h(N)$ that the set $\RP_{N,x,i}$ has a natural density. We then say that $N$ has \textbf{tame residual dimension at $x$.} We define the \textbf{upper residual dimension of $N$ at $x$} to be the minimal index $i_0$, denoted $\dim_{\RFU}(N,x)$, such that $\delta(\RP_{N,x,i_0}) > 0$.  We write $UR_{N,x} = \RP_{N,x,t_0}$ and call  $UR_{N,x}$ the \textbf{set of prime numbers that realize the upper residual dimension of $N$ at $x$}.
	
	Now suppose that $N$ has tame residual dimension at all primitive elements $x \in \sqrt{\ga_c(N)}$. We denote the \textbf{upper residual dimension of $N$} as
	$$
	\dim_{\RFU}(N) \bdef \text{max}\{ \dim_{\RFU}(N,x) \: | \: x \in \sqrt{\ga_c(N)} \text{ such that } x \text{ is primitive}\}.
	$$
	When $N$ is an infinite, finitely generated nilpotent such that $N/T(N)$ has tame residual dimension at every primitive element $x \in \sqrt{\ga_c(N/T(N))}$ where $c$ is the step length of $N/T(N)$, we denote $\dim_{\RFU}(N) \bdef \dim_{\RFU}(N/T(N))$.
\end{defn}

The following lemma forms the basis for the definition of tame residual dimension.
\begin{lemma}\label{abelian_tame}
	Let $A$ be a torsion free, finitely generated abelian group, and let $x$ be a primitive element. Then $\RP_{A,x,1} = \mathbb{P}$.
\end{lemma}	
\begin{proof}
	Let $p$ be prime number. Since $x$ is primitive, there exists a generating basis $\set{z_i}_{i=1}^{h(A)}$ for $A$ such that $z_1 = x$. Letting $B = \innp{z_i}_{i=2}^{h(A)}$, we note that $A / B \cong \Z$ and that $\pi_B(x) \neq 1$. By taking the natural map  $\map{\varphi}{\Z}{\Z / p \Z}$ given by reduction modulo $p$, one can see that $\varphi \circ \pi_B(x) \neq 1$. Thus, $p \in \RP_{A,x,1}$ as desired.
\end{proof}
	
\begin{defn}
Suppose that $N$ is a torsion free, finitely generated nilpotent group of step length $c$. If $N$ is abelian, we have by Lemma \ref{abelian_tame} that $N$ has tame residual dimension at all primitive elements and that any Mal'tsev basis for $N$ is a residually tame generating basis for $\ga_1(N) = N$. Thus, we say that all torsion free, finitely generated abelian groups have \textbf{tame residual dimension.} Now suppose that $N$ has step length $c > 1$. Suppose that $N / \sqrt{\ga_c(N)}$ has tame residual dimension, that there exists a residually tame basis for $\sqrt{\ga_c(N)}$, and that $N$ has tame residual dimension at all primitive elements $x \in \sqrt{\ga_c(N)}$.  We say that $N$ has \textbf{tame residual dimension}.

If $N$ is an infinite, finitely generated nilpotent group such that $N/T(N)$ has tame residual dimension, we say that $N$ has \textbf{tame residual dimension.}
\end{defn}

We now relate the upper residual dimension of a torsion free, finitely generated nilpotent group to torsion free quotients of lower step length.
\begin{prop}\label{residual_dim_lower_step_quotient}
	Let $N$ be a torsion free, finitely generated nilpotent group of step length $c$, and suppose that $N$ has tame residual dimension.  If $N$ is abelian, then $\dim_{\RFU}(N) = 1$. Otherwise, letting $M = \sqrt{\ga_c(N)}$, we then have that $\dim_{\RFU}(N/M) \leq \dim_{\RFU}(N)$.
\end{prop}
\begin{proof}
Since the first statement is clear, we may assume that $N$ has step length $c > 1$. Let $a \in \sqrt{\ga_{c-1}(N/M)}$ be a primitive element. There exists a primitive element $b \in \sqrt{\ga_{c-1}(N)}$ such that $\pi_M(b) = a$. Since $b \in \sqrt{\ga_{c-1}(N)}$, there exists a natural number $s$ such that $b^s \in \ga_{c-1}(N)$. Thus, there exists an element $g \in N$ such that $[b^s,g]$ is a nontrivial element of $\ga_c(N)$. Hence, there exists some primitive element $x \in M$ and a natural number $k$ such that $x^k = [b^s,g]$. Let $p \in \UR_{N,x}$ be a prime number that does not divide $k$, and let $\map{\psi}{N}{Q}$ be a $p$-witness of $x$. Since $\text{Ord}_Q(\psi(x)) \nmid k$, we have that $\psi(x^k) \neq 1$. In particular, we have that $\psi([b^s,g]) \neq 1$, and therefore, $\psi(b) \notin \ga_c(Q)$. Thus, we have an induced homomorphism $\map{\bar{\psi}}{N/M}{Q/\ga_c(Q)}$ such that $\bar{\psi}(\pi_M(b)) \neq 1$. Therefore, it follows that $\bar{\psi}(a) \neq 1$, and thus, by definition, if $\varphi:N/M \to K$ is a finite $p$-witness for $a$, we have that $
|K| \leq |Q|= p^{\dim_{\RFU}(N,x)}.
$ Letting $A_i = \RP_{N/M,a,i} \cap \UR_{N,x}$ for $1 \leq i \leq \dim_{\RFU}(N,x)$, we see that the previous inequality implies that $\UR_{N,x} = \sqcup_{i=1}^{\dim_{\RFU}(N,x)}A_i$. If $\delta(\RP_{N/M,a,i}) = 0$ for each $i$, then we would have that $\delta(A_i) = 0$. In particular, we would have that $\delta(\UR_{N,x}) = \sum_{i=1}^{\dim_{\RFU}(N,x)} \delta(A_i) = 0$ which is a contradiction. Thus, there exists an index $1 \leq i_0 \leq \dim_{\RFU}(N/M,a)$ such that $\delta(\RP_{N/M,a,i_0}) > 0$. Thus,
$
\dim_{\RFU}(N/M,a) \leq \dim_{\RFU}(N,x) \leq \dim_{\RFU}(N).
$
This inequality holds for all primitive elements of $\sqrt{\ga_{c-1}(N/M)}$; hence, we have that 
$
\dim_{\RFU}(N/M) \leq \dim_{\RFU}(N).
$\end{proof}
We finish by giving an explicit inequality that relates the values $\dim_{\RFU}(N)$ and $\dim_{\RFL}(N)$.
\begin{prop}\label{upper_vs_lower}
	Suppose that $N$ is a torsion free, finitely generated nilpotent group that has tame residual dimension. Then 
	$
	\dim_{\RFL}(N) \leq \dim_{\RFU}(N).
	$
\end{prop}
\begin{proof}
	Let $x$ be a primitive nontrivial element of $\sqrt{\ga_c(N)}$. Since $\delta(\UR_{N,x}) > 0$, we have that $|\UR_{N,x}| = \infty$. Thus, $\dim_{\RFL}(N,x) \leq \dim_{\RFU}(N,x)$. Hence, we have that $\dim_{\RFL}(N,x) \leq \dim_{\RFU}(N)$. Since this inequality holds for all primitive elements of $\sqrt{\ga_c(N)}$, we have that $\dim_{\RFL}(N) \leq \dim_{\RFU}(N)$.
\end{proof}
\subsection{Accessible residual dimension}
For a torsion free, finitely generated nilpotent group $N$ that has tame residual dimension, one may be interested in when $
\dim_{\RFL}(N) = \dim_{\RFU}(N)
$. In this case, we would be able to obtain a precise asymptotic characterization of the growth of residual finiteness of a finitely generated nilpotent group. Therefore, we introduce the following definition and proposition.
\begin{defn}\label{accessible_residual}
	Let $N$ be a torsion free, finitely generated nilpotent group of step length $c$ that has tame residual dimension. Let $x \in \sqrt{\ga_{c}(N)}$ be a primitive  element such that $\dim_{\RFU}(N,x) > 1$. We say that $N$ has \textbf{low dimensional vanishing at $x$} if $|\RP_{N,x,i}| < \infty$ for $1 \leq i < \dim_{\RFU}(N,x)$. If $\dim_{\RFU}(N,x) = 1$, we will always say that $N$ has low dimensional vanishing at $x$.
	
	Suppose that $N$ has step length $c = 1$. Since Lemma \ref{abelian_tame} implies that all  torsion free, finitely generated abelian groups have low dimensional vanishing for all primitive elements $x$, we will say that $N$ has \textbf{accessible residual dimension.}  Now suppose that $N$ has step length $c > 1$. If $N/ \sqrt{\ga_c(N)}$ has accessible residual dimension and $N$ has low dimension residual vanishing at all primitive elements $x \in \sqrt{\ga_c(N)}$, we say that $N$ has \textbf{accessible residual dimension}.

	If $N$ is an infinite, finitely generated nilpotent group such that $N / T(N)$ has accessible residual dimension, we say that $N$ has accessible residual dimension.
\end{defn}

We are able to associate a natural number to any infinite, finitely generated nilpotent group with accessible residual dimension that captures the degree of the polynomial in logarithm of the residual finiteness growth.
\begin{defn}
	Let $N$ be a torsion free, finitely generated nilpotent group that has accessible residual dimension. It is straightfoward to see that $\dim_{\RFL}(N) = \dim_{\RFU}(N)$. We denote their common value as $\dim_{\ARF}(N)$ and call this value the \textbf{accessible residual dimension of $N$}. When $N$ is an infinite, finitely generated nilpotent group with accessible residual dimension, we set $\dim_{\ARF}(N) \bdef \dim_{\ARF}(N / T(N))$.
\end{defn}

\subsection{Residual finiteness of nilpotent groups with torsion}
Before proceeding to the upper and lower bounds for residual finiteness of finitely generated nilpotent groups, we have the following proposition. This proposition and its proof are originally found in \cite[Proposition 4.4]{Pengitore_1}. It relates the effective residual finiteness of an infinite, finitely generated nilpotent group to its torsion free quotient. We produce the proof as it is short and for the sake of completeness.

\begin{prop}\label{torsion_to_torsion_free}
	Let $N$ be an infinite, finitely generated nilpotent group. Then
	$
	\Farb_{N}(n) \approx \Farb_{N / T(N)}(n).
	$
\end{prop}
\begin{proof}
	We proceed by induction on $|T(N)|$, and note that the base case is clear. Hence, we may assume that $|T(N)| > 1$. We have that the group morphism $\map{\pi_{Z(T(N))}}{N}{N / Z(T(N))}$ is surjective with kernel given by $Z(T(N))$ which is a finite central subgroup. Since finitely generated nilpotent groups are linear,  \cite[Lemma 2.4]{Bou_Rabee_Kaletha} implies that
	$
	\Farb_{N}(n) \approx \Farb_{N/Z(T(N))}(n).
	$
	Since 
	$
	(N / Z(T(N))) / T(N / Z(T(N))) \cong N / T(N),
	$
	the induction hypothesis implies that
	$
	\Farb_{N}(n) \approx \Farb_{N / T(N)}(n).
	$
\end{proof}

With the above proposition, we may prove the following theorem.
\begin{thmn}[\ref{torsion_free_abelian}]
	Let $N$ be an infinite, finitely generated nilpotent group such that $N/T(N)$ is abelian. Then
	$$
	\Farb_{N}(n) \approx \log(n).
	$$
\end{thmn}
\begin{proof}
	Proposition \ref{torsion_to_torsion_free} implies that $\Farb_N(n) \approx \Farb_{N/T(N)}(n)$. We also have that \cite[Corollary 2.3]{Bou_Rabee10} implies $\Farb_{N/T(N)}(n) \approx \log(n)$. Therefore,
	$
	\Farb_{N/T(N)}(n) \approx \log(n).
	$
\end{proof}
\section{Lower bounds for residual finiteness of finitely generated nilpotent groups}\label{lower_bound_section}
We restate Theorem \ref{lower_bound_rf_result} for the convenience fo the reader. 
\begin{thmn}[\ref{lower_bound_rf_result}]
	Let $N$ be an infinite, finitely generated nilpotent group such that $N / T(N)$ has step length $c > 1$. There exists a natural number $\dim_{\RFL}(N)$ such that $\dim_{\RFL}(N) \geq c + 1$ and where $$
	\pr{\log(n)}^{\dim_{\RFL}(N)} \precsim \Farb_{N}(n).
	$$
\end{thmn}
\begin{proof}

We start by assuming that $N$ is torsion free which implies that $N$ has step length $c > 1$. Let $x \in \sqrt{\ga_{c}(N)}$ be an element $N$ such that $\dim_{\RFL}(N,x) = \dim_{\RFL}(N)$, and let $k$ be the minimal natural number satisfying $x^k \in \ga_{c}(N)$. We may choose a finite Mal'tsev generating subset  $\set{x_i}_{i=1}^{t}$ for $N$  such that $x_1 = x$. Proposition \ref{lower_bound_residual_dimension} implies that $\dim_{\RFL}(N,x) \geq c + 1$, and the definition of $\dim_{\RFL}(N,x)$ implies that the set 
	$
	A = \bigcup_{i=1}^{\dim_{\RFL}(N,x) - 1}\RP_{N,x,i}
	$
	is finite. Thus, we may write $A = \set{q_1 < q_2 < \cdots < q_\ell}$ where $q_i$ are prime numbers for all $i$. Let $\set{p_j}_{j=1}^{\infty}$ be an enumeration of the set  $\set{p \in \LR_{N,x} \: | \:  p > \text{max}\set{q_\ell,k}}$, and let $m_j = \pr{\lcm\set{1,\cdots, p_{j} - 1}}^{\dim_{\RFL}(N) + 1}.$ We claim that $\{x^{k \: m_j}\}_{j=1}^{\infty}$ is the desired sequence. That implies we must show that
	$
	\D_N(x^{k \: m_j}) \approx \pr{\log(\|x^{k \: m_j}\|)}^{\dim_{\RFL}(N)}
	$
for all $j$.

We see \cite[3.B2]{Gromov} implies that 
	$
	\|x^{k \: m_j}\| \approx  m_j^{1/c},
	$
	and additionally, the Prime Number Theorem \cite[1.2]{Tenenbaum} implies that $\log(m_j) \approx p_{j}$. Subsequently, $\log(\|x^{k \: m_j}\|) \approx p_{j}$, and thus, $\pr{\log(\|x^{k \: m_j}\|)}^{\dim_{\RFL}(N)} \approx p_{j}^{\dim_{\RFL}(N)}.$ That implies we have two tasks. We first need to demonstrate that there exists a surjective group morphism $\map{f}{N}{P}$ to a finite group $P$ such that $|P| = p_{j}^{\dim_{\RFL}(N)}$ and where $f(x^{k \: m_j}) \neq 1$. Secondly, we need to demonstrate that if given a surjective group morphism $\map{\varphi}{N}{Q}$ to a finite group where $|Q| < p_{j}^{\dim_{\RFL}(N)}$, then $\varphi(x^{k \: m_j}) = 1$.
	
	Let $\map{\psi_j}{N}{P_j}$ be a $p_j$-witness of $x$. By definition, $\psi_j(x) \neq 1$, $|P_j| = p_{j}^{\dim_{\RFL}(N)}$, and $\text{Ord}_P(\psi_j(x)) = p_{j}$. Since $p_{j}^t \nmid k \: m_j$, we have that $\psi(x^{k \: m_j}) \neq 1$ as desired.
	
	Now suppose that $\map{\varphi}{N}{Q}$ is a surjective group morphism to a finite group where $|Q| < p_{j}^{\dim_{\RD}(N)}$. Since $\varphi(x) = 1$ implies that $\varphi(x^{k \: m_j}) = 1$, we may assume $\varphi(x) \neq 1$. Result \cite[Theorem 2.7]{Hall_notes} implies we may assume that $|Q| = q^\la$ where $q$ is some prime number.  For notational simplicity, we let $s_q$ be the natural number such that a $q$-witness of $x$ has order $q^{s_q}$.
	
If $q^\la < p_{j}$, then by construction, we know that $|Q| \mid m_j$, and since the order of an element divides the order of the group, we have that $\text{Ord}_Q(\varphi(x)) \mid m_j$. Therefore, $\varphi(x^{k \: m_j}) = 1$.
	
If $q < p_{j}$ and $p_{j} < q^\la < p_{j}^{\dim_{\RFL}(N)}$, then there exists a natural number $\nu$ such that 
	$
	q^\nu < p_{j} < q^{\nu + 1}.
	$
	Thus, we have that
	$$
	q^{\nu \: \dim_{\RFL}(N)} < p_{j}^{\dim_{\RD}(N)} < q^{(\nu + 1) \: \dim_{\RFL}(N)}.
	$$
	Subsequently, we may write $\la = \nu \: \ell + r$ where $\ell \leq \dim_{\RFL}(N)$ and $0 \leq r < \nu$.  By assumption, $q^\nu < p_{j}$. Therefore, we must have that $q^\nu \mid \lcm\set{1,\cdots, p_{j} - 1}$ and thus,  $$(q^\nu)^{\ell} \mid \pr{\lcm\set{1,\cdots, p_{i_j}-1}}^{\dim_{\RFL}(N)}.$$ Hence, $q^\la \mid m_j$. Since the order of $\varphi(x)$ divides $|Q|$, we have that $\varphi(x^{k \: m_j}) = 1$.
	
	Now suppose that $q > p_{j}$ and that $s_q \geq \dim_{\RFL}(N)$. Since $\varphi(x) \neq 1$, we have that $\la \geq s_q$. In particular, we have that 
	$$
	|Q| = q^\la \geq q^{s_q} \geq q^{\dim_{\RFL}(N)} \geq p_{j}^{\dim_{\RFL}(N)}.
	$$
	Hence, we may disregard this case.
	
	For the final case, we may assume that $q > p_j$ and $s_q< \dim_{\RFL}(N)$. By definition, $q \in A$; however, by the choice of prime numbers $p_j$ we have that $p_j > q$, which is a contradiction. Therefore, this case is not possible, and we may ignore it.
	
	Therefore, $\D_N(x^{k \: m_j}) \approx \pr{\log(\|x^{k \: m_j}\|)}^{\dim_{\RFL}(N)}$, and thus, $\pr{\log(n)}^{\dim_{\RFL}(N)} \precsim \Farb_{N}(n).$ Additionally, Proposition \ref{lower_bound_residual_dimension} implies that $\dim_{\RFL}(N) \geq c + 1$. 
	
	When $N$ is an infinite, finitely generated nilpotent group where $|T(N)|>1$, we have by the above arguments that $\pr{\log(n)}^{\dim_{\RFL}(N / T(N))}  \precsim \Farb_{N / T(N)}(n)$. We also have that $\dim_{\RFL}(N / T(N)) \geq c + 1$ where $c$ is the step length of $N/T(N)$.  Proposition \ref{torsion_to_torsion_free} implies that $\Farb_{N}(n) \approx \Farb_{N / T(N)}(n)$; moreover, we have that $\dim_{\RFL}(N) = \dim_{\RFL}(N / T(N))$. Therefore, $\dim_{\RFL}(N) \geq c + 1$ and 
	$
	\pr{\log(n)}^{\dim_{\RFL}(N)} \precsim \Farb_{N}(n). 
	$
\end{proof}

\section{Upper bounds for residual finiteness for finitely generated nilpotent groups}\label{upper_bounds_section}
The main goal of this section is to prove the following theorem.
\begin{thmn}[\ref{upper_bounds_residual_function}]
	Let $N$ be an infinite, finitely generated nilpotent group.Then
	$$
	\Farb_N(n) \precsim \pr{\log(n)}^{\psi_{\RD}(N)}.
	$$ Now suppose that $N$ has tame residual dimension. Then there exists a natural number $\dim_{\RFU}(N)$ satisfying $\dim_{\RFU}(N) \leq \psi_{\RD}(N)$ such that 
	$$
	\Farb_{N}(n) \precsim \pr{ \log( n)}^{\dim_{\RFU}(N)}.
	$$ 
	If $\dim_{\RFU}(N) < \psi_{\RD}(N),$ then $\Farb_{N}(n)$ grows strictly slower than what is predicted by \cite[Theorem 1.1]{Pengitore_1}.
\end{thmn}

In order to make sense of this statement, we need to define the constant $\psi_{\RD}(N)$ found in \cite[Theorem 1.1]{Pengitore_1}. We start with the following proposition which is originally found in \cite[Proposition 3.1]{Pengitore_1} and which we recall for the reader's convenience. 
\begin{prop}\label{one_dim_center_primitive_element}
	Let $N$ be a torsion free, finitely generated nilpotent group, and suppose that $z$ is a primitive central element. There exists a normal subgroup $H \nsub N$ such that $N / H$ is a torsion free, finitely generated nilpotent group where $\innp{\pi_{H}(z)} \cong Z(N / H)$.
\end{prop}
\begin{proof}
	We proceed by induction on Hirsch length to produce a normal subgroup $H \nsub N$ such that $\innp{\pi_H(z)} \cong Z(N/H)$ and where $N/H$ is a torsion free, finitely generated nilpotent group. Since the statement is evident for when $N \cong \Z$, we may suppose that $h(N) > 1$. If $h(Z(N)) = 1$, then as in the base case, the statement is evident. Therefore, we may assume that $h(Z(N)) > 1$.
	
	There exists a generating basis $\set{x_i}_{i=1}^{h(Z(N))}$ for $Z(N)$ such that $x_1 =z$. If we consider the subgroup given by $M = \innp{x_i \: | \: i \leq 2 \leq h(Z(N))}$, induction implies that there exists a normal subgroup $H / M \nsub N / M$ such that $(N/M)/(H/M)$ is a torsion free, finitely generated nilpotent group such that  $\innp{\pi_{H/M}(\pi_M(z))} \cong Z((N/H) / (H/M))$. The third isomorphism theorem implies that $(N/M) / (H/M) \cong N/H$, and subsequently, $Z(N/H) \cong Z((N/M)/(H/M))$. Therefore, it is evident that $H \nsub N$ is our desired normal subgroup.
\end{proof}

With the above proposition, we introduce the following definition.
\begin{defn}
	Let $N$ be a torsion free, finitely generated nilpotent group of step length $c$ with a primitive element $x \in \sqrt{\ga_{c}(N)}$. Proposition \ref{one_dim_center_primitive_element} implies that the value
	$$
	\dim_{\R}(N,x) \bdef \text{min}\set{h(N / H) \: | \: H \text{ satisfies Proposition \ref{one_dim_center_primitive_element} for $x$}}
	$$
	is bounded above by $h(N)$. We refer to the value $\dim_{\R}(N,x)$ as the \textbf{real residual dimension of $N$ with respect to $x$}. Letting $\mathcal{J}$ be the collection of primitive elements $x \in \sqrt{\ga_c(N)}$ such that there exists a natural number $k$ and elements $a \in \ga_{c-1}(N)$ and $b \in N$ such that $[a,b] = x^k$,  there exists a primitive element $z \in \sqrt{\ga_{c}(N)}$ such that
	$$
	\dim_{\R}(N,z) \bdef \text{max}\set{\dim_{\R}(N,x) \: | \: x \in \mathcal{J}}
	$$
	We refer to $\dim_{\R}(N,z)$ as the \textbf{real residual dimension of $N$} and denote it as $\psi_{\RD}(N)$. When $N$ is an infinite, finitely generated nilpotent group, we denote $\psi_{\RD}(N) \bdef \psi_{\RD}(N / T(N))$.
\end{defn}

If $N$ satisfies $h(Z(N)) = 1$, then one can see that the definition of $\psi_{\RD}(N)$ implies that $\psi_{\RD}(N) = h(N)$.

With the above in mind, we now compare the values $\dim_{\RFU}(N)$ and $\psi_{\RD}(N)$ for torsion free, finitely generated nilpotent groups.
\begin{prop}\label{residual_dim_vs_real_res_dim}
	Let $N$ be a torsion free, finitely generated nilpotent group that has tame residual dimension. Then $\dim_{\RFU}(N) \leq \psi_{\RD}(N)$.
\end{prop}
\begin{proof}
	If $h(Z(N)) = 1$, we have that $\psi_{\RD}(N) = h(N)$. Thus, it follows that $\dim_{\RFU}(N) \leq \psi_{\RD}(N)$. Hence, we may assume that $h(Z(N)) > 1$. Let $x \in \sqrt{\ga_{c}(N)}$ be a primitive  element. Proposition \ref{one_dim_center_primitive_element} implies there exists a normal subgroup $H \nsub N$ such that $h(N / H) = \dim_{\R}(N,x)$. Letting $p$ be a prime number, we have that $\pi_{N^p \cdot H}(x) \neq 1$ and that $|N / N^p \cdot H| = p^{\dim_{\R}(N,x)}$. By definition, we have that there exists a $p$-witness  $\psi:N \to Q$ for $x$ and where $|Q| \leq p^{\dim_{\R}(N,x)}$. Since $h(Z(N)) > 1$, we have that $\RP_{N,x,i} = \emptyset$ for $\dim_{\R}(N,x) < i \leq h(N)$. Thus, there exists an index $i_0$ where $1 \leq i_0 \leq \psi_{\RD}(N)$ such that $\UR_{N,x} = \RP_{N,x,i_0}$. Therefore, $\dim_{\RFU}(N,x) \leq \psi_{\RD}(N)$. Since this is true independent of primitive elements in $\sqrt{\ga_{c}(N)}$, we have that $\dim_{\RFU}(N) \leq \psi_{\RD}(N)$.
\end{proof}

We have the following technical proposition.
\begin{prop}\label{real_residual_dim_inequality}
	Let $N$ be a torsion free, finitely generated nilpotent group. If $N$ is abelian, then $\psi_{\RD}(N) = 1$. If $N$ has step size $c > 1$, then $
	\psi_{\RD}(N / \sqrt{\ga_c(N)}) \leq \psi_{\RD}(N).
	$
\end{prop}
\begin{proof}
Since the first statement is clear, we may assume that $N$ has step length $c > 1$. Let $M = \sqrt{\ga_c(N)}$, and let $g \in \sqrt{\ga_{c-1}(N/M)}$ be a primitive element. There exists a primitive element $x \in \sqrt{\ga_{c-1}(N)}$ such that $\pi_M(x) = g$. Since $x \in \sqrt{\ga_{c-1}(N)}$, there exists a natural number $m$ such that $x^m \in \ga_{c-1}(N)$. Thus, there exists an element $y \in N$ such that $[x^m,y]$ is a nontrivial element in $\ga_c(N)$. Hence, there exists an element $z \in M$ and a natural number $s$ such that $z^s = [x^m,y]$. Let $H \nsub N$ satisfy Proposition \ref{one_dim_center_primitive_element} for $z$ in $N$ such that $\dim_{\R}(N,z) = h(N/H)$. By construction, we have that $\pi_H([x^m,y]) \neq 1$ since $\pi_H(z)$ generates $Z(N/H)$. In particular, we have that $\pi_H(x) \notin \sqrt{\ga_c(N/H)}$. Hence, we have an induced surjective group morphism $\overline{\pi_H} \colon N / M \to (N/H) / \sqrt{\ga_c(N/H)}$ so that $\overline{\pi_H}(\pi_M(x)) \neq 1$. Thus, we have that $\overline{\pi_H}(g) \neq 1$. By construction, we have $N/M \cdot H$ is a torsion free, finitely generated nilpotent group. By following the proof of Proposition \ref{one_dim_center_primitive_element} we may find a normal subgroup $K / M \cdot H \nsub N/ M \cdot H$ so that $(N/M \cdot H) / (K / M \cdot H)$ is a quotient  of $N/M$ that satisfies Proposition \ref{one_dim_center_primitive_element} for $g$. In particular, we may write:
$$
\dim_\R(N/M,g) \leq h(N/M \cdot H) \leq \dim_{\R}(N,z) \leq \psi_{\RD}(N).
$$
Since $g$ is an arbitrary primitive element of $\sqrt{\ga_{c-1}(N/M)}$, we have that $\psi_{\RD}(N/M) \leq \psi_{\RD}(N).$
%	Let $M = \sqrt{\ga_c(N)}$. Let $g$ be a primitive element in $\sqrt{\ga_{c-1}(N/M)}$. There exists a primitive element $x \in \sqrt{\ga_{c-1}(N)}$ such that $\pi_M(x) = g$. Subsequently, there exists a natural number $m$ such that $x^m \in \ga_{c-1}(N)$. Thus, there exists an element $y \in N$ such that $[x^m,y]$ is a nontrivial element in $\ga_c(M)$. Hence, there exists a primitive element $z \in \sqrt{\ga_c(N)}$ and a natural number $k$ such that $z^k = [x^m,y]$. Let $H \nsub N$ satisfy Proposition \ref{one_dim_center_primitive_element} for $z$ such that $\dim_{\R}(N,z) = h(N/H)$. By definition, we have that $\pi_H(z) \neq 1$, and thus, it follows that $\pi_H(x^m) \neq 1$. Since $\pi_H(x^m) \notin Z(N/H)$ and $\pi_H(M) \leq Z(N/H)$, we have that $\pi_H(x^m) \notin \pi_H(M)$. In particular, we have that $\pi_H(x) \notin \pi_H(M)$. Let $K/M \nsub N/M$ satisfy Proposition \ref{one_dim_center_primitive_element} for $\pi_M(x)$ where $\dim_{\R}(N/M,\pi_M(x)) = h((N/M) / (K/M))$. 
%
%
%By direct calculation, we have that $H \leq K$; thus, we have that $h(N/H) \geq h((N/M)/(K/M))$. Since $\pi_M(x) = g$, we may write
%	$
%	\dim_\R(N/M,g) \leq \dim_{\R}(N,z) \leq \dim_{\R}(N).
%	$
%	Since $g$ is an arbitrary primitive element of $\sqrt{\ga_{c-1}(N/M)}$, we have by definition that
%	$
%	\dim_{\R}(N/M) \leq \dim_{\R}(N). 
%	$
\end{proof}

\noindent We now proceed to the proof of Theorem \ref{upper_bounds_residual_function}.
\begin{proof} 
	Let us first assume that $N$ is a torsion free, finitely generated nilpotent group of step length $c$. We proceed with the proof of the first statement. 
	
	Let $S = \set{x_i}_{i=1}^{h(N)}$ be a Mal'tsev basis, and for simplicity, let $M = \sqrt{\ga_{c}(N)}$. Let $g = \prod_{i=1}^{h(N)} x_i^{\al_i}$ be a nontrivial element of word length at most $n$. We proceed by induction on step length, and observe that the base case follows from \cite[Corollary 2.3]{Bou_Rabee10}. Thus, we may assume that $N$ has step length $c > 1$.
	
	If $\pi_{M}(g) \neq 1,$ then induction implies that there exists a surjective group morphism to a finite group $\map{\varphi}{N / M}{Q_1}$ such that $\varphi(\bar{g}) \neq 1$ and where $|Q_1| \leq C_1 \pr{\log(C_1 \: n)}^{\dim_{\RFU}(N / M)}$ for some constant $C_1 > 0$. Proposition \ref{real_residual_dim_inequality} implies that $\psi_{\RD}(N / M) \leq \psi_{\RD}(N)$. Since $\map{\varphi \circ \pi_M}{N}{Q_1}$ satisfies $\varphi \circ \pi_M(g) \neq 1$, we have that
	$
	\D_{N}(g) \leq C_1 \pr{\log(C_1 \: n)}^{\psi_{\RD}(N)}.
	$
	
	Now suppose that $\pi_M(g) = 1$. That implies that we may write $g = \prod_{i=1}^{h(M)}x_i^{\al_i}$, and since $\|g\| \leq n$, Lemma \ref{coord_bound} implies that there exists a constant $C_2 > 0$ such that $|\al_i| \leq C_2 \: n^c$ for all $i$. There exists a primitive element $z = \prod_{i=1}^{h(M)}x_i^{\beta_i}$ and a nonzero integer $m$ so that $z^m = g$. Let $H \nsub N$ satisfy Proposition \ref{one_dim_center_primitive_element} for $z$ in $N$ such that $\dim_{\R}(N,z) = h(N/H)$. 
	Since $\set{x_1,\cdots,x_{h(M)}}$ are central elements, we have for all $i$ that $\beta_i \: m = \al_i$. In particular, we have that $|m| \leq C_2 \: n^c$. By the Prime Number Theorem, there exists a constant $C_3 > 0$ such that $p \leq C_3 \: \log(C_3 \: n)$ and where $p \nmid m$. By construction, $\pi_{H \cdot N^p}(g) = \pi_{H \cdot N^p}(z^m) \neq 1$. Thus, there exists a constant $C_4 > 0$ such that 
	$$
	\D_N(g) \leq |N/H \cdot N^p| = p^{\dim_{\R}(N,x)} \leq p^{\psi_{\RD}(N)} \leq C_4 \pr{\log(C_4 \: n)}^{\psi_{\RD}(N)}.
	$$
	Since all possibilities have been covered, we have that
	$
	\Farb_{N}(n) \precsim \pr{\log(n)}^{\psi_{\RD}(N)}
	$
	when $N$ is an arbitrary, torsion free, finitely generated nilpotent group.
	
	We now assume that $N$ is a torsion free, finitely generated nilpotent group with tame residual dimension. Thus, we take $S = \{x_i\}_{i=1}^{h(N)}$ as a Mal'tsev basis where $\{x_i\}_{i=1}^{h(M)}$ is a residually tame generating basis for $M$.  As before, we may proceed by induction on step length, and observe that the base case follows from \cite[Corollary 2.3]{Bou_Rabee10}. Thus, we may assume that $c > 1$. If $\pi_{M}(g) \neq 1$, then by the inductive hypothesis, there exists a surjective group morphism $\map{\psi}{N/M}{Q_1}$ such that $\psi(\pi_{M}(g)) \neq 1$ and where $|Q_1| \leq \pr{\log(C_5 \:  n)}^{\dim_{\RFU}(N/M)}$
for some constant $C_5 > 0$. Proposition \ref{residual_dim_lower_step_quotient} implies that $\dim_{\RFU}(N/M) \leq \dim_{\RFU}(N),$ and thus, $\D_N(g) \leq \pr{C_5 \: \log(C_5 \: n)}^{\dim_{\RFU}(N)}.$
	
	Therefore, we may assume that $\pi_M(g) = 1$; hence, we may write $g = \prod_{i=1}^{h(M)}x_i^{\al_i}$. For each $1 \leq i \leq h(M)$, the Prime Number Theorem implies that there exists a constant $C_{6,i}$ such that for all $k \in \N$, there exists a prime $p \in \UR_{N,x_i}$ such that $p \leq C_{6,i} \log(C_{6,i} |k|)$ and where $p \nmid k$. We let $C_6 = \text{max}\{C_{6,i} \: | \: 1 \leq i \leq h(M)\}$. Proposition \ref{coord_bound} implies that there exists a constant $C_7 > 0$ such that $|\al_i| \leq C_7 \: n^c$ for all $1 \leq i \leq h(M)$. There exists a primitive element $z \in M$ such that $z^m = g$ for some nonzero integer $m$.  Writing $z = \prod_{i=1}^{h(M)}x_i^{\beta_i}$ and noting that $\set{x_1,\cdots,x_{h(M)}}$ are central elements, we have that $\beta_i \: m = \al_i$ for all $i$, and thus, it follows that $|m| \leq |m \: \beta_i| = |\al_i| \leq C_7 \: n^c$. The Prime Number Theorem implies that there exists a prime $p \in \UR_{N,x_{i_0}}$ such that $p \leq C_8\: \log(C_8 \: n)$ for some prime $p$  where $p \nmid |m|$. 
	
	By assumption, we have that $\{x_i\}_{i=1}^{h(\sqrt{\ga_c(N)})}$ is a residually tame basis for $\sqrt{\ga_c(N)}$, and thus, by defintion, we have that there exists an index $1 \leq i_0 \leq h(\sqrt{\ga_c(N)}$ such that $p \nmid \al_{i_0}$ and where there exists a $p$-witness $\map{\varphi}{N}{P}$ of $x_{i_0}$ such that $\varphi(z) \neq 1$. Moreover, we have that $|P| \leq p^{\dim_{\RFU}(N)}$. Since $p \nmid m$, we have that $\text{Ord}_P(\varphi(z)) \nmid |m|$. In particular, $\varphi(g) = \varphi(z^m) \neq 1$. Thus,
	$
	\D_N(g) \leq \pr{C_8 \: \log(C_8 \: n)}^{\dim_{\RFU}(N)}.
	$
	Thus,
	$
	\Farb_N(n) \precsim \pr{ \log(  n)}^{\dim_{\RFU}(N)}.
	$
	
	Now suppose that $N$ is an infinite, finitely generated nilpotent group where $|T(N)| > 1$. By definition,
	$
	\psi_{\RD}(N) = \psi_{\RD}(N/T(N)).
	$ Thus, we have that $\Farb_{N/T(N)}(n) \precsim \pr{\log(n)}^{\psi_{\RD}(N)}$ by the above arguments.  Proposition \ref{torsion_to_torsion_free} implies that
	$
	\Farb_N(n) \approx \Farb_{N/T(N)}(n) \precsim \pr{\log(n)}^{\psi_{\RD}(N)}.
	$
	
	Now suppose that $N$ is an infinite, finitely generated nilpotent group that has tame residual dimension where $|T(N)| > 1$. As before, we note that $\dim_{\RFU}(N) = \dim_{\RFU}(N/T(N))$ by definition. The above arguments and Proposition \ref{torsion_to_torsion_free} imply that we may write
	$
\Farb_N(n) \approx \Farb_{N/T(N)}(n) \precsim \pr{ \log( n)}^{\dim_{\RFU}(N)}.
	$
\end{proof}

\section{Residual finiteness of finitely generated nilpotent groups with accessible residual dimension}\label{accessible_section}
We now prove Theorem \ref{accessible_result}.
\begin{thmn}[\ref{accessible_result}]
	Let $N$ be an infinite, finitely generated nilpotent group such that $N / T(N)$ has step length $c > 1$, and suppose that $N$ has accessible residual dimension. Then there exists a natural number $\dim_{\ARF}(N)$ such that $c + 1 \leq \dim_{\ARF}(N) \leq \psi_{\RD}(N)$ and where
	$$
	\Farb_{N}(n) \approx \pr{\log(n)}^{\dim_{\ARF}(N)}.
	$$
\end{thmn}
\begin{proof}
	By Theorem \ref{upper_bounds_residual_function}, we have that $\Farb_{N / T(N)}(n) \precsim \pr{\log(n)}^{\dim_{\RFU}(N / T(N))}$ and that $\dim_{\RFU}(N/T(N)) \leq \psi_{\RD}(N/T(N))$. By Theorem \ref{lower_bound_rf_result}, we have that $\pr{\log(n)}^{\dim_{\RFL}(N/T(N))} \precsim \Farb_{N / T(N)}(n)$. Moreover, we have the following inequality $c +1  \leq \dim_{\RFL}(N/T(N))$. Since $$\dim_{\ARF}(N /T(N)) = \dim_{\RFL}(N / T(N)) = \dim_{\RFU}(N / T(N)),$$ we have that $\Farb_{N / T(N)}(n) \approx \pr{\log(n)}^{\dim_{\ARF}(N / T(N))}$. Proposition \ref{torsion_to_torsion_free} implies that $\Farb_{N}(n) \approx \Farb_{N / T(N)}(n)$. Since  $\dim_{\ARF}(N) = \dim_{\ARF}(N / T(N))$ and $\psi_{\RD}(N) = \psi_{\RD}(N/T(N))$, it follows that $c + 1 \leq \dim_{\ARF}(N) \leq \psi_{\RD}(N)$. Thus, 
	 $
	 \Farb_{N}(n) \approx \pr{\log(n)}^{\dim_{\ARF}(N)}.
	 $
\end{proof}

\section{Effective separability of Filiform nilpotent groups}\label{separability_heisenberg}
We start this section by defining Filiform nilpotent groups.
\begin{defn}
Suppose that $N$ is a torsion free, finitely generated nilpotent group of Hirsch length $h \geq 3$. If $N$ has step length $h-1$, then we say that $N$ is a \textbf{Filiform nilpotent group}. 
\end{defn}

A collection of Filiform groups are groups given by the presentation where $h \geq 3$:
$$
\mathcal{F}_h = \innp{x_1, \cdots, x_{h} \: | \: [x_1, x_i] = x_{i+1} \text{ for } 2 \leq i \leq h-1 \text{ and all other commutators are trivial }}.
$$
In particular, this class of Filiform groups includes the $3$-dimensional integral Heisenberg group. 

Let $N$ be a Filiform nilpotent group of Hirsch length $h$. The torsion free quotient of the abelianization of $N$ is isomorphic to $\Z^2$. We also have that the Hirsch length of $\sqrt{\ga_i(N)}$ is $h - i$ for $i>1$ and that $\sqrt{\ga_{h-1}(N)} \cong \Z$. 

We start by calculating the order of a finite $p$-witness for a primitive element for all but finitely many primes.

\begin{lemma}\label{heis_p_things}
	Let $N$ be a Filiform nilpotent group of Hirsch length $h \geq 3$ and let $p$ be prime number. Let $z \in \sqrt{\ga_{h-1}(N)}$ be a primitive element where $z^k \in \ga_{h-1}(N)$. If $\varphi:N \to P$ is a $p$-witness for $z$ where $p \nmid k$, then $|P| = p^h$.
\end{lemma}
\begin{proof}
	Under the group morphism $\map{\pi_p}{N}{N / N^p}$, we have that $\pi_p(z) \neq 1$ and $|N / N^p| = p^h$. Therefore, $|Q| \leq p^h$. Now suppose that $\map{\varphi}{N}{Q}$ is a surjective group morphism to a finite $p$-group where $\varphi(z) \neq 1$. Since $\gcd(k,p) = 1$, we have that $\text{Ord}_Q(\varphi(z)) \nmid k$. Subsequently, $\varphi(z^k) \neq 1$. Thus, $\varphi(\ga_{h-1}(N))$ is nontrivial, and hence, $Q$ has step length $h-1$. Lemma \ref{p_dim_step_length} implies that $|Q| \geq p^h$. Hence, $|Q| = p^h$.
\end{proof}

As a consequence, we are able to show that all Filiform nilpotent groups have accessible residual dimension.
\begin{prop}
Let $N$ be a Filiform nilpotent group of Hirsch length $h$. Then $N$ has accessible residual dimension.
\end{prop}
\begin{proof}
We proceed by induction on step length of $N$. Since $N/ \sqrt{\ga_{h-1}(N)}$ is either a lower dimensional Filiform nilpotent group or abelian, we have by induction that $N / \sqrt{\ga_{h-1}(N)}$ has accessible residual dimension. Since $h(\sqrt{\ga_c(N)}) = 1$, we have that $N$ has a residually tame generating basis for $\sqrt{\ga_c(N)}$. Let $z \in \sqrt{\ga_{h-1}(N)}$ be a primitive element. There exists an integer $k$ such that $z^k \in \ga_{h-1}(N)$. If we let $A = \{ p \in \mathbb{P} \: | \: p \text{ divides } k\}$, we have that $A$ is finite. In particular, $\mathbb{P} \backslash A$ has positive natural density; moreover, we have that $\mathbb{P} \backslash A \subseteq  \RP_{N,x,h}$. Thus, we have that $\RP_{N,z,h}$ is all but finitely many primes which implies that it has positive natural density. Finally, one can see that $\cup_{i=1}^{h-1} \RP_{N,z,i} \subseteq A$. Thus, either $\RP_{N,z,i}$ has positive natural density or is finite for all $i$.
\end{proof}

As a consequence, we are able to precisely compute $\Farb_N(n)$ when $N$ is a Filiform nilpotent group.
\begin{cor}
Let $N$ be a Filiform nilpotent group of Hirsch length $h \geq 3$. Then $\Farb_{N}(n) \approx (\log(n))^h$. In particular, if $N$ is the $3$-dimensional integral Heisenberg group, then $\Farb_N(n) \approx (\log(n))^3$.
\end{cor}
\bibliography{bib}
\bibliographystyle{plain}
\end{document}